\documentclass[10pt,reqno]{amsart} 
\usepackage{amssymb}
\usepackage[cp1250]{inputenc}   
\usepackage{hyperref}
\usepackage[all]{xy}            
\usepackage{verbatim}
\usepackage{tikz}
    \usetikzlibrary{shadows,matrix,arrows}
\usepackage{mathtools}
\usepackage{longtable}

\newtheorem{Thm}{Theorem}[section]
\newtheorem{Prp}[Thm]{Proposition}

\newtheorem{Crl}[Thm]{Corollary}
\newtheorem{Cnj}[Thm]{Conjecture}

\theoremstyle{definition}

\newtheorem{Exp}[Thm]{Example}

\theoremstyle{remark}
\newtheorem{Rmk}[Thm]{Remarks}
\numberwithin{equation}{section}

\def\opp{\mathrm{op}}       
\def\Ker{\mathrm{Ker}}
\def\Im{\mathrm{Im}}

\def\N{\mathbb{N}}
\def\Z{\mathbb{Z}}
\def\Q{\mathbb{Q}}
\def\N{\mathbb{N}}
\def\R{\mathbb{R}}
\def\C{\mathbb{C}}

\def\HP{\mathrm{HP}}

\newcommand{\icases}[7]{#7\{\!\begin{smallmatrix}
                                #5#1\hfill;  & ~#5#2\hfill &\!\!\\[#6]
                                #5#3\hfill;  & ~#5#4\hfill &\!\!
                              \end{smallmatrix}}

\begin{document}
\title{(Co)Homology of Poset Lie Algebras}
\author{Leon Lampret}
\address{Institute for Mathematics, Physics and Mechanics, Ljubljana, Slovenia\vspace{-6pt}}
\address{Department of Mathematics, University of Ljubljana, Slovenia}
\email{leon.lampret@imfm.si}
\author{Aleš Vavpetič}
\address{Department of Mathematics, University of Ljubljana, Slovenia}
\email{ales.vavpetic@fmf.uni-lj.si}
\date{April 29, 2015}
\keywords{algebraic/discrete Morse theory, homological algebra, chain complex, acyclic matching, solvable Lie algebra, triangular matrices, torsion table, algebraic combinatorics}
\subjclass{17B56, 13P20, 13D02, 55-04, 18G35, 58E05}

\begin{abstract}
We investigate the (co)homological properties of two classes of Lie algebras that are constructed from any finite poset: the solvable class $\frak{gl}^\preceq$ and the nilpotent class $\frak{gl}^\prec$. We confirm the conjecture \cite[1.16.(1), p.141]{citeJollenbeckADMTACA} that says: every prime power $p^r\!\leq\!n\!-\!2$ appears as torsion in $H_\ast(\frak{nil}_n;\Z)$, and every prime power $p^r\!\leq\!n\!-\!1$ appears as torsion in $H_\ast(\frak{sol}_n;\Z)$. If $\preceq$ is a bounded poset, then the (co)homology of $\frak{gl}^\preceq$ is \emph{torsion-convex}, i.e. if it contains $p$-torsion, then it also contains $p'$-torsion for every prime $p'\!<\!p$.
\par We obtain new explicit formulas for the (co)homology of some families over arbitrary fields. Among them are the solvable non-nilpotent analogs of the Heisenberg Lie algebras from \cite{citeCairnsJamborCHLAOFFC}, the 2-step Lie algebras from \cite{citeArmstrongCairnsJessupEBNFNLA}, strictly block-triangular Lie algebras, etc. The resulting generating functions and the combinatorics of how they are obtained are interesting in their own right.
\par All this is done by using AMT (algebraic Morse theory \cite{citeLampretVavpeticCLAAMT}, \cite{citeSkoldbergMTFAV}, \cite{citeJollenbeckADMTACA}). This article serves as a source of examples of how to construct useful acyclic matchings, each of which in turn induces compelling combinatorial problems and solutions. It also enables graph theory to be used in homological algebra.
\end{abstract}\vspace{-8mm}

\maketitle

\subsection*{Background} Let $\frak{nil}_n\!\leq\!\frak{sol}_n\!\leq\!\frak{gl}_n$ be Lie algebras of all ((strictly) triangular) $n\!\times\!n$ matrices over a field $K$. By the Ado-Iwasawa theorem, every finite-dimensional Lie $K$\!-algebra $\frak{g}$ admits an embedding into some $\frak{gl}_n$. By Engel's theorem, $\frak{g}$ is nilpotent iff it admits an embedding into some $\frak{nil}_n$. By Lie's theorem, $\frak{g}$ is solvable iff it admits an embedding into some $\frak{sol}_n$, when $K$ is algebraically closed of characteristic $0$ or more that $\dim\frak{g}$. Lie algebras $\frak{gl}_n$, $\frak{sol}_n$, $\frak{nil}_n$ over $\C$ induce the classical Lie groups $G\!L_n\!=\! \big\{A\!\in\!\C^{n\times n};\, \det A\!\neq\!0\big\}$ and $SOL_n\!=\! \big\{A\!\in\!GL_n;\, A_{ij}\!=\!0\text{ for }i\!>\!j\big\}$ and $N\!I\!L_n\!=\! \big\{A\!\in\!SOL_n;\, A_{ii}\!=\!1 \big\}$ via the exponential map \smash{$X\!\mapsto\!\sum_{k\in\N}\!\frac{X^k}{k!}$}.
\par Nilpotent and solvable Lie algebras are less wild than general ones. For instance, there are 9 cases and 2 infinite families of nonisomorphic Lie algebras of dimension $\leq3$ over $\R$ or $\C$, by the Bianchi classification (i.e. continuum many isomorphism classes). But there are only 2 cases and 2 families of nonisomorphic solvable Lie algebras of dimension $\leq3$ over \emph{any} field, by \cite{citeGraafCSLA}. Also, there are just 9 cases of nonisomorphic nilpotent Lie algebras of dimension $\leq5$ over \emph{any} field, by \cite{citeGraafC6NLAFC2}.
\par In this article, we concentrate on those Lie subalgebras of $\frak{nil}_n$ and $\frak{sol}_n$ which admit a basis consisting of matrix units $e_{ij}$. These are precisely  $\frak{gl}_n^\prec$ and $\frak{gl}_n^\preceq$ that arise from posets, and they are combinatorial in nature. Their structure constants are $\pm1$ or $0$. Special cases of these families include $\frak{nil}_n$, $\frak{sol}_n$, block strictly triangular matrices, the Heisenberg family from \cite{citeCairnsJamborCHLAOFFC}, the 2-step family from \cite{citeArmstrongCairnsJessupEBNFNLA}, etc. We are interested in the (co)homology of $\frak{gl}_n^\preceq$ and $\frak{gl}_n^\prec$ from a combinatorial viewpoint.

\subsection*{Results} In this article, the following facts are proved. Every convex subposet induces a direct summand on (co)homology. A poset and its opposite induce the same (co)homology. There holds Poincar\'{e} duality: $H^k(\frak{gl}_n^\prec;\Z) \cong H_{\dim\!-k}(\frak{gl}_n^\prec;\Z)$ for all $k$, but $H^k(\frak{gl}_n^\preceq;\Z) \ncong H_{\dim\!-k}(\frak{gl}_n^\preceq;\Z)$. There holds $C_\star(\frak{gl}_n^\prec)=\bigoplus_{w\in\Z^n}\!C_{[w]}$ and $C_{\star,p}(\frak{gl}_n^\prec)=\bigoplus_{w\in(p\Z)^n}\!C_{[w]}$, where $C_{[w]}$ is the subcomplex spanned by all wedges with weight vector $w$. This is very useful for computational purposes: it enabled us to obtain $H_\ast(\frak{nil}_7;\Z)$ on a home computer, even though $\dim\Lambda^{\!\ast}\frak{nil}_7\!=\! 2^{\binom{7}{2}}\!=\! 2097152$.
\par For any $a,b\!\in\![n]$ and any $m\!\in\!\N$ such that the interval $[a,b]$ in $\preceq$ contains more than $m$ elements, there is a direct summand $\Z_m$ in $H_\ast(\frak{gl}_n^\preceq;\Z)$. If $\preceq$ has a largest and a least element, then $H_\ast(\frak{gl}_n^\preceq;\Z)$ contains $p$-torsion iff $p\!<\!n$. There is a direct summand $\Z_{n\!-\!2}$ in $H_\ast(\frak{nil}_n;\Z)$. If $H_\ast(\frak{gl}_n^\preceq;\Z)$ does not contain $2$-torsion, then its Hasse diagram is a forest of height $\leq1$, hence $H_\ast(\frak{gl}_n^\preceq;\Z)$ is torsion-free. If chain complex $C_{\star,p}(\frak{gl}_n^\prec)$ is nonempty, then $C_{\star,p'}(\frak{gl}_n^\prec)$ is nonempty for every prime $p'\!<\!p$.
\par For five families of posets $\preceq$ of height $1$ or $2$, the generating functions for graded $K$\!-modules $H_\ast(\frak{gl}_n^\preceq;K)$ are computed in \ref{HGT1} and \ref{HGT2}. Our approach has an advantage: we obtain explicit generators for (co)homology, not just a formula for dimension. This enabled us to compute the cohomology algebras w.r.t. cup products.

\subsection*{Conventions} Throughout this article, $R$ will denote a commutative unital ring. Letter $p$ will always denote a prime number, and $[n]\!=\!\{1,2,\ldots,n\}$. We omit the wedge signs $\wedge$ in the basis elements of $\Lambda^{\!k}\frak{g}$. For any partial ordering $\preceq$ we denote by $\Gamma_\preceq$ its Hasse diagram (acyclic digraph) and by $\overline{\Gamma}_\preceq$ the undirected version (simple graph). For a Lie $R$-algebra $\frak{g}$ we denote its Chevalley chain complex by $C_\star(\frak{g};R)\!=\!\Lambda^{\!\star}\frak{g}$ with boundary $\partial\,x_1\!\ldots x_k\!=\! \sum_{r<s}(-1)^{r+s}[x_r,x_s]x_1\!\ldots\widehat{x_r}\!\ldots \widehat{x_s}\!\ldots x_k$. The associated digraph is denoted by $\Gamma_{C_\star}$ (its vertices are the basis elements of modules $C_k(\frak{g};R)$, and edges correspond to nonzero entries in the boundary matrices). This digraph is used to define Morse matchings and use AMT (algebraic Morse theory, see \cite[1.1]{citeLampretVavpeticCLAAMT}). We use the notations from that formulation.\vspace{4mm}

\vspace{4mm}
\section{Posets}
In this section, we define poset Lie algebras and explore their general properties.\\[1mm]

\par Let $\sim$ be a transitive relation on a set $I$ (so it is a subset of $I\!\times\!I$). Let $\frak{gl}^\sim(R)$ be the free $R$-module on symbols $e_{ij}$ for $i\!\sim\!j$ (i.e. the free $R$-module on the elements of $\sim$). Then $[e_{ij},e_{kl}]=\delta_{jk}e_{il}\!-\!\delta_{il}e_{kj}$ is a well-defined bracket that makes $\frak{gl}^\sim$ a Lie $R$-algebra. It is isomorphic to the Lie subalgebra of $\frak{gl}_I(R)$ (all column-finitary $I\!\times\!I$ matrices with the commutator bracket) spanned by basis matrices $e_{ij}$ with $i\!\sim\!j$. If $I$ has finite cardinality $n$, then we denote it by $\frak{gl}_n^\sim$. The Lie subalgebras of $\frak{gl}_n$ that admit a module basis $\mathcal{B}\!\subseteq\! \{e_{ij}; i,j\!\in\![n]\}$ are precisely the Lie algebras $\frak{gl}_n^\sim$. Every $\frak{gl}^\sim$ is a $\Z$-graded Lie algebra via $\deg e_{ij}\!=\!i\!-\!j$.
\par When the relation $\sim$ is antisymmetric and reflexive, we denote it by $\preceq$; when $\sim$ is antisymmetric and irreflexive, we denote it by $\prec$. Any isomorphism of posets $(I,\preceq)\!\rightarrow\!(J,\sqsubseteq)$ induces isomorphisms of Lie algebras $\frak{gl}^\preceq\!\longrightarrow\!\frak{gl}^\sqsubseteq$ and $\frak{gl}^\prec\!\longrightarrow\!\frak{gl}^\sqsubset$ that send $e_{ij}\!\longmapsto\!e_{f(i)f(j)}$. Any poset morphism, however, does not canonically induce a Lie algebra morphism: if $I\!=\!\{a,b,c,d\}$ and $J\!=\!\{1,2,3\}$ are linearly ordered and $f$ sends $a\!\mapsto\!1, b\!\mapsto\!2, c\!\mapsto\!2, d\!\mapsto\!3$, then $f[e_{ab},e_{cd}]=0 \neq e_{13}=[f(e_{ab}),f(e_{cd})]$.
\par For any poset $\preceq$, our $\frak{gl}_n^\prec$ is an ideal of $\frak{gl}_n^\preceq$ and thus induces an exact sequence\vspace{-1mm} $$0\!\longrightarrow\!\frak{gl}_n^\prec \!\longrightarrow\! \frak{gl}_n^\preceq \!\longrightarrow\! \frak{dgn}_n \!\!\longrightarrow\!0,\vspace{-1mm}$$ where $\frak{dgn}_n$ is the abelian Lie algebra of diagonal matrices. Thus the Hochschild-Serre spectral sequence can be used to obtain information of $\frak{gl}_n^\prec$ from $\frak{gl}_n^\preceq$. Additionally, there is an exact sequence\vspace{-1mm}
$$0\!\longrightarrow\!\frak{sl}_n^\preceq \!\overset{\scriptstyle\subseteq}{\longrightarrow}\! \frak{gl}_n^\preceq \!\overset{\mathrm{tr}}{\longrightarrow}\! R \!\longrightarrow\!0,\vspace{-1mm}$$
and the ideal $\frak{sl}_n^\preceq$ of matrices with zero trace was recently studied in \cite{citeCollGerstenhaberCLSPPA}.\\[2mm]

The \emph{opposite} poset of $\preceq$ is the poset $\preceq^\opp$ with the same underlying set, but $i\!\preceq^\opp\!j$ iff $i\!\succeq\!j$). A \emph{subposet} of $(J,\sqsubseteq)$ is a subset $I\!\subseteq\!J$ with the induced ordering $\preceq\,=\,\sqsubseteq\cap (I\!\times\!I)$. A subposet $(I,\preceq)$ of $(J,\sqsubseteq)$ is \emph{convex} when $a,b\!\in\!I$ and $x\!\in\!J$ and $a\!\sqsubseteq\!x\!\sqsubseteq\!b$ implies $x\!\in\!I$, i.e. if it is closed for taking intervals $[a,b]\!=\! \{x\!\in\!J; a\!\sqsubseteq\!x\!\sqsubseteq\!b\}$.

\begin{Prp}\label{subsetopposite} {\rm\textbf{a)}} $H_\ast\big(\frak{gl}^{\preceq^\opp}\big) \cong H_\ast\big(\frak{gl}^{\preceq}\big)$.\\
{\rm\textbf{b)}} If $\preceq$ is a convex subposet of $\sqsubseteq$, then $C_\star(\frak{gl}^{\sqsubseteq})=C_\star(\frak{gl}^{\preceq})\!\oplus\!\ldots$.
\end{Prp}
Thus a poset and its opposite induce isomorphic homologies, and every convex subposet induces a direct summand on homologies.
\begin{proof}\textbf{a)} The transposition map $\frak{gl}^{\preceq^\opp} \!\longrightarrow\! \frak{gl}^\preceq$ that sends $e_{ji}\!\mapsto\!e_{ji}^t\!=\!e_{ij}$ is a module morphism that satisfies $[\mathbf{a},\mathbf{b}]^t\!=\!-[\mathbf{a}^t\!,\mathbf{b}^t]$ for all $\mathbf{a},\mathbf{b}\!\in\!\frak{gl}_n$. The induced module morphisms $\Lambda^{\!k}\frak{gl}^{\preceq^\opp} \!\longrightarrow\! \Lambda^{\!k}\frak{gl}^\preceq$ anticommute with the boundary morphisms. Hence they induce morphisms on all homology modules. The same can be said for the transposition map in the other direction, which produces the inverse morphisms.\\
\textbf{b)} For any wedge $v\!=\!e_{a_1b_1}\!\!\cdots\!e_{a_kb_k} \!\in\! \Lambda^{\!k}\frak{gl}^\sqsubseteq$ the set of indices $\epsilon(v)\!=\!\{a_1,b_1,\ldots,a_k,b_k\}$ is a subposet of $(I,\sqsubseteq)$. Since the only nonzero brackets are of the form $[e_{ab},e_{bc}]\!=\!e_{ac}$, we see that the set of maximal and minimal indices of $v$ equals the set of maximal and minimal indices of every summand of $\partial(v)$. Thus the whole chain complex $C_\star(\frak{gl}^\sqsubseteq)$ is a direct sum of subcomplexes spanned by wedges with fixed sets of maximal and minimal indices. In particular, when $(I,\preceq)$ is convex in $(J,\sqsubseteq)$, the wedges $\{v; \epsilon(v)\!\subseteq\!I\}\!\subseteq\!C_\star(\frak{gl}^\sqsubseteq)$ span the chain complex $C_\star(\frak{gl}^\preceq)$.
\end{proof}
The analogous two statements for Lie algebras $\frak{gl}_n^\prec$ also hold.

\begin{Exp}\label{induction} Let $\preceq$ be the total order on $[n]$, so that $\frak{gl}_n^\prec\!=\!\frak{nil}_n$ is the Lie algebra of all upper triangular matrices. The convex subposets of the interval $[1,n]$ are smaller intervals $[i,j]$. For $i,j\!\in\![n]$ let $C_{i,j}$ be the subcomplex of $C_\star(\frak{nil}_n)$ spanned by $\{v;\, \min\epsilon(v)\!=\!i,\max\epsilon(v)\!=\!j\}$. Notice that $C_{i,j}\cong C_{1,j-i+1}$. We have $C_\star(\frak{nil}_n)=\bigoplus_{i\leq j}\!C_{i,j}=\frac{C_{1,n}\,\oplus\,\bigoplus_{\!1<i\leq j}\!C_{i,j} \,\oplus\, \bigoplus_{\!i\leq j<n}\!C_{i,j}}{\bigoplus_{\!1<i\leq j<n}\!C_{i,j}}$. Therefore $$H_\ast(\frak{nil}_n)\cong\frac{H_\ast\!(C_{1,n})\,\oplus\,H_\ast\!(\frak{nil}_{n\!-\!1})^2}{H_\ast\!(\frak{nil}_{n\!-\!2})}$$ where $C_{1,n}$ is the subcomplex of all wedges that contain both $1$ and $n$ as indices. The analogous statement for the Lie algebra $\frak{gl}_n^\preceq\!=\!\frak{sol}_n$ also holds.
\end{Exp}\vspace{1mm}

\begin{Rmk}Disjoint unions of Hasse diagrams induce tensor products:
$$C_\star(\frak{gl}^{\preceq\sqcup\preceq'}) \cong C_\star(\frak{gl}^{\preceq}) \!\otimes\! C_\star(\frak{gl}^{\preceq'})\text{ \;\;and\;\; } C_\star(\frak{gl}^{\prec\sqcup\prec'}) \cong C_\star(\frak{gl}^{\prec}) \!\otimes\! C_\star(\frak{gl}^{\prec'}).$$ Thus from now on, we only consider connected Hasse diagrams, and ignore all mirror images. Note that nonisomorphic posets may induce isomorphic Lie algebras, e.g.\vspace{-2mm} $$\Gamma_\preceq\!: \begin{tikzpicture}[baseline=(m.center)]\matrix (m) [matrix of math nodes, row sep=8pt, column sep=8pt] {2&4\\ 1&3\\};
\draw(m-2-1)--(m-1-1); \draw(m-2-2)--(m-1-2);\end{tikzpicture}\text{ \;\;\;\;and\;\;\;\; }
\Gamma_\sqsubseteq\!: \begin{tikzpicture}[baseline=(m.center)]\matrix (m) [matrix of math nodes, row sep=8pt, column sep=8pt] {2&&3\\ &1&\\};
\draw(m-2-2)--(m-1-1); \draw(m-2-2)--(m-1-3);\end{tikzpicture}\vspace{-2mm}$$
induce $\frak{gl}_4^\prec(R)\cong\frak{gl}_3^\sqsubset(R)\cong R^2$, though $\frak{gl}_4^\preceq(R)\ncong\frak{gl}_3^\sqsubseteq(R)$, and \vspace{-2mm}
$$\Gamma_\preceq\!: \begin{tikzpicture}[baseline=(m.center)]\matrix (m) [matrix of math nodes, row sep=8pt, column sep=8pt] {2&4\\ 1&3\\};
\draw(m-2-1)--(m-1-1); \draw(m-2-1)--(m-1-2);\draw(m-2-2)--(m-1-2);\end{tikzpicture}\text{ \;\;\;\;and\;\;\;\; }
\Gamma_\sqsubseteq\!: \begin{tikzpicture}[baseline=(m.center)]\matrix (m) [matrix of math nodes, row sep=8pt, column sep=8pt] {2&3&4\\ &1&\\};
\draw(m-2-2)--(m-1-1); \draw(m-2-2)--(m-1-2); \draw(m-2-2)--(m-1-3);\end{tikzpicture}\vspace{-2mm}$$
induce \;$H_\ast(\frak{gl}_4^\prec;R)\cong H_\ast(\frak{gl}_4^\sqsubset;R)$ \;and\; $H_\ast(\frak{gl}_4^\preceq;R)\cong H_\ast(\frak{gl}_4^\sqsubseteq;R)$.
\end{Rmk}\vspace{2mm}


\begin{Prp} \!\!$H^k(\frak{gl}_n^\prec;\Z) \cong H_{N\!-k}(\frak{gl}_n^\prec;\Z)$ for all $k$, where $N\!=\!|\!\!\prec\!|\!=\!\dim\frak{gl}_n^\prec$.
\end{Prp}
\noindent However, $H^k(\frak{gl}_n^\prec;\Z)\ncong H_{N\!-k}(\frak{gl}_n^\prec;\Z)$ for all $k$, where $N\!=\!|\!\preceq\!|\!=\!\dim\frak{gl}_n^\preceq\!>\!n$, since by \cite[5.8]{citeLampretVavpeticCLAAMT} the free part is $\Z^{\binom{n}{k}}$\!, which is symmetric iff $N\!=\!n$ iff $\preceq\,=\!\{(i,i); i\!\in\![n]\}$.
\begin{proof} By Hazewinkel's theorem \cite{citeHazewinkelDTCLA}, it suffices to prove that $\frak{gl}_n^\prec$ is unimodular, i.e. for every $a\!\in\!\frak{gl}_n^\prec$ the linear map $[a,-]\!: \frak{gl}_n^\prec\!\rightarrow\!\frak{gl}_n^\prec$ has trace $0$. Since $a\!=\!\sum_{r\prec s}\alpha_{rs}e_{rs}$ and $\mathrm{tr}[a,-]=\sum_{r\prec s}\alpha_{rs}\,\mathrm{tr}[e_{rs},-]$, it suffices to prove $\mathrm{tr}[e_{rs},-]\!=\!0$. Indeed, in the basis $\{e_{ij}; i\!\prec\!j\}$, the only elements that are not annihilated are $e_{st}$ and $e_{tr}$, but they are mapped to $e_{rt}$ and $-e_{ts}$, so the coefficient of $[e_{rs},e_{ij}]$ at $e_{ij}$ is $0$.
\end{proof}
Thus $H_\ast(\frak{gl}_n^\prec)$ obeys Poincar\'{e} duality: free part appears symmetrically ($FH_k\cong FH_{N\!-k}$) and torsion appears symmetrically shifted by one ($T_pH_k\cong T_pH_{N\!-k-1}$).

\vspace{4mm}
\section{Generating functions}
\vspace{2mm}
Let $K$ be a field and $M\!=\!\bigoplus_{k\in\N} M_k$ a graded $K$-module such that every $M_k$ is finite dimensional. The \emph{Hilbert-Poincar\'{e} series} of $M$ is the formal power series\vspace{-0mm} $$\textstyle{\HP_{\!M}(t)=\sum_k\dim_KM_k\,t^k\in\Z[[t]].}\vspace{-0mm}$$
There holds $\HP_{\!M/N}=\HP_{\!M}-\HP_{\!N}$, $\HP_{\!M\oplus N}=\HP_{\!M}+\HP_{\!N}$, $\HP_{\!M\otimes N}=\HP_{\!M}\cdot\HP_{\!N}$.
\par For a Lie $K$-algebra $\frak{g}$ we denote the Hilbert-Poincar\'{e} series of its homology by $\HP_{\!C\!_\star\!(\frak{g})}(t)= \sum_k\dim_K\!H_k(\frak{g})\, t^k\!\in\!\Z[[t]]$, when $\dim_K\!H_k(\frak{g})\!<\!\infty$ for all $k$. Thus $\HP$ contains all information about the (co)homology of $\frak{g}$ over $K$. There always holds $\HP_{\!C\!_\star\!(\frak{g})}(-1)= \sum_k(-1)^k\dim_K\!H_k(\frak{g})= \sum_k(-1)^k\dim_K\!\Lambda^{\!k}(\frak{g})= \sum_k(-1)^k\binom{\dim\frak{g}}{k}= (1-1)^{\dim\frak{g}}=0$. If $K\!\leq\!K'$ is a field extension, then  $\HP_{\!C_\star(\frak{g};K)} \!=\! \HP_{\!C_\star(\frak{g};K')}$ by the universal coefficient theorem, so it suffices to compute over fields $\Q$ and $\Z_p$.
\par We obtained some information about $\HP$ for Lie algebras $\frak{gl}_n^\preceq$ and $\frak{gl}_n^\prec$. If $K$ has characteristic $0$ or more than $n$, then by \cite[5.8]{citeLampretVavpeticCLAAMT} there holds \smash{$\HP_{\!C\!_\star\!(\frak{gl}_n^\preceq)}(t)\!=\!(1\!+\!t)^n$}. Also, $\HP_{\!C\!_\star\!(\frak{gl}^{\preceq\sqcup\preceq'})} \!=\! \HP_{\!C\!_\star\!(\frak{gl}^{\preceq})} \cdot \HP_{\!C\!_\star\!(\frak{gl}^{\preceq'})}$ and $\HP_{\!C\!_\star\!(\frak{gl}^{\prec\sqcup\prec'})} \!=\! \HP_{\!C\!_\star\!(\frak{gl}^{\prec})} \cdot \HP_{\!C\!_\star\!(\frak{gl}^{\prec'})}$.
\par In the Chevalley chain complex $C_\star(\frak{\frak{gl}_n^\preceq})$, the \emph{weight vector} of a wedge $v$ of basis matrices $e_{ij}$ is the vector $w_v\!=\!(w_1,\ldots,w_n)\in\Z^n$ in which $w_i$ is the number of times $i$ appears as a right index $e_{ai}\!\in\!v$ minus the number of times it appears as a left index $e_{ib}\!\in\!v$. In symbols, $w_i=|\{a\!\in\![n]; e_{ai}\!\in\!v\}|\!-\!|\{b\!\in\![n]; e_{ib}\!\in\!v\}|$. Thus $-n\!<\!w_i\!<\!n$ for all $i$, and if $i$ is a minimal (resp. maximal) element of $\preceq$, then $w_i\!\leq\!0$ (resp. $w_i\!\geq\!0$). If our field $K$ has characteristic $p$, then by \cite[6.1]{citeLampretVavpeticCLAAMT} we have\vspace{-1mm} $$\HP_{\!C\!_\star\!(\frak{gl}_n^\preceq)}(t)=(1\!+\!t)^n\cdot\HP_{\!C_{\star,p}(\frak{gl}_n^\prec)}(t)\vspace{-1mm}$$
where \emph{$p$-complex} $C_{\star\!,p}(\frak{gl}_n^\prec)$ is the chain subcomplex (direct summand) of $C_\star(\frak{gl}_n^\prec)$, spanned by all wedges $v$ with $w_v\!\in\!(p\Z)^n$, i.e. weights are noninvertible in \smash{$\Z_p\!=\!\frac{\Z}{p\Z}$}. We call $C_{\star\!,p}(\frak{gl}_n^\prec)$ \emph{empty} when it is spanned by just the empty wedge $\emptyset$.
\par Since every edge in the digraph $\Gamma_{C_\star(\frak{gl}^\prec)}$ is of the form $\begin{smallmatrix}\ldots e_{ax}e_{xb}\ldots\\[-1pt] \downarrow\\[-2pt] \ldots e_{ab}\ldots\end{smallmatrix}$, edges preserve the weight vector, i.e. all summands of $\partial(v)$ have the same weight vector $w_v$.

\begin{Crl} $C_\star(\frak{gl}_n^\prec)=\bigoplus_{w\in\Z^n}\!C_{[w]}$ and $C_{\star,p}(\frak{gl}_n^\prec)=\bigoplus_{w\in(p\Z)^n}\!C_{[w]}$, where $C_{[w]}$ is the subcomplex spanned by all wedges with weight vector $w$.
\end{Crl} This holds over any ring, and it is very useful for computational purposes, since each boundary $\partial_k$ is a huge matrix in block diagonal form, and we only have to deal with the very many much smaller blocks. Even if we are working with sparse matrices, it is much more effective to deal with each $C_{[w]}$ separately.
\par Similarly, $C_\star(\frak{gl}_n^\preceq)=\bigoplus_{w\in\Z^n}\!C_{[w]}$ over any ring. Here the isolated vertices of $\Gamma_{C_\star}$ (the wedges of diagonals $e_{ii}$, which by \cite[5.8]{citeLampretVavpeticCLAAMT} generate the free part in $H_\ast(\frak{gl}_n^\preceq;\Z)$) constitute the subcomplex $C_{[0]}$, and for any $w\!=\!(w_1,\ldots,w_n)\!\neq\!0$ we notice that $\gcd(w_1,\ldots,w_n)\!=\!1$ implies $C_{[w]}\simeq0$, by \cite[6.1]{citeLampretVavpeticCLAAMT}. Thus many $C_{[w]}$ can be omitted from the direct sum without changing the homology of the complex.
\par Of course, the usefulness of this only becomes apparent when we are able to generate each $C_{[w]}$ (without first generating the whole $C_\star(\frak{gl}_n^\preceq)$). At present, we only have an algorithm for constructing all $C_{[w]}$ in $C_\star(\frak{sol}_n)$ and $C_\star(\frak{nil}_n)$.

\begin{Rmk}\label{DFT} Recall that the \emph{height} of a poset $\preceq$ is $\max\{r; \exists x_0\!\prec\!x_1\!\prec\!\ldots\!\prec\!x_r\}$, i.e. the length of the longest chain in the poset. In sections \ref{HGT1} and \ref{HGT2}, we will see that when height is 1 or 2, the $p$-complex is small enough to allow direct computation.
There, we shall also be using a lemma from the theory of discrete Fourier transforms: given $\sum_{k\in\N}c_kt^k\!=\!f(t)\in\C[[t]]$ and $p\!\in\!\N$, the sum of every $p$-th term is $\sum_{k\in p\N}c_kt^k=\frac{1}{p}\sum_{i\in[p]}f(\epsilon^it)$ where $\epsilon\!=\!e^{2\pi\mathbf{i}/p}$. More generally, 
$\textstyle{\sum_{k\in p\N}c_kt^{jk+l}=\frac{t^l}{p}\sum_{i\in[p]}f(\epsilon^it^j).}$
\end{Rmk}

\begin{Rmk}\label{cupproduct} Over a field, any chain complex $C_\star$ contains submodules $C'_\star$ with $\partial'\!=\!0$ and $C''_\star\simeq0$, such that $C_\star= C'_\star\oplus C''_\star$ and therefore $H_\ast(C_\star)=C'_\ast$. Indeed, any submodule is a direct summand, so $\Im\partial_{k+1}\leq\Ker\partial_k\leq C_k$ implies there exist $D'_k,D''_k\leq C_k$ with $\Im\partial_{k+1}\!\oplus\!D_k'\!=\!\Ker\partial_k$ and $\Im\partial_{k+1}\!\oplus\!D_k'\!\oplus\!D_k''\!=\!C_k$, hence $\partial|_{D'_k}\!=\!0$ and \smash{$\partial\!:D''_k\!\overset{\cong}{\longrightarrow}\!\Im\partial_k$}, so $C'_k\!=\!D'_k$ and $C''_k\!=\!\Im\partial_{k+1}\!\oplus\!D_k''$ suffice.
\par Finding the submodule $C'_\star$ can make the computation of cup products possible. If $C_\star$ is the Chevalley complex, then for $\alpha\!\in\!H^i(\frak{g})$ and $\beta\!\in\!H^j(\frak{g})$ the rule\vspace{-1mm} $$(\alpha\!\smile\!\beta)(x_1\!\ldots x_{i+j})= \!\!\sum_{\substack{\scriptscriptstyle\pi\in S_{i+j},\,\pi_1<\ldots<\pi_i,\\\scriptscriptstyle \pi_{i+1}<\ldots<\pi_{i+j}}}\!\!\! \mathrm{sgn}\pi\,\alpha(x_{\pi_1}\!\ldots x_{\pi_i})\, \beta(x_{\pi_{i+1}}\!\ldots x_{\pi_{i+j}})\vspace{-1mm}$$ makes $H^\ast(\frak{g})$ a graded-commutative algebra. If $C'_\star\!=\! H_\ast(C_\star)$ admits a finite basis consisting of wedges $\{v_i; i\!\in\!I\}$, then $C'^\star\!=\!H^\ast(C_\star)$ has a dual basis $\mathcal{B}\!=\!\{\chi_{v_i}; i\!\in\!I\}$, and there holds $\chi_{v_i}\!\!\smile\!\chi_{v_j}=\chi_{v_i\wedge v_j}$ with $\chi_v\!=\!0$ iff $v\!\notin\!\mathcal{B}$.
\end{Rmk}


\vspace{4mm}
\section{Homology tables}
It is always desirable to have concrete examples of objects under inspection. We provide two tables for the homology of families that are studied in later sections. In each of the two cases, the first row is a sequence of Hasse diagrams $\Gamma\!_\preceq$, and under each of them is the graded module $H_\ast(\frak{gl}^\preceq)$.

\vspace{3pt}{\small\label{3.tableCompleteBipartite}
\begin{longtable}[c]{@{\hspace{0pt}}l@{\hspace{5pt}} l@{\hspace{5pt}} l@{\hspace{5pt}} l@{\hspace{5pt}} l@{\hspace{5pt}} l@{\hspace{5pt}} l@{\hspace{5pt}}}
$\!\!\!\!\!\begin{tikzpicture}\matrix (m) [matrix of math nodes, row sep=17pt, column sep=5pt]{
3&4\\
1&2\\ };
\draw(m-2-1)--(m-1-1); \draw(m-2-1)--(m-1-2); \draw(m-2-2)--(m-1-1); \draw(m-2-2)--(m-1-2);\end{tikzpicture}$\!\!\!&
$\!\!\!\!\!\begin{tikzpicture}\matrix (m) [matrix of math nodes, row sep=17pt, column sep=2pt]{
4&&5\\
1&2&3\\ };
\draw(m-2-1)--(m-1-1);\draw(m-2-1)--(m-1-3); \draw(m-2-2)--(m-1-1);\draw(m-2-2)--(m-1-3); \draw(m-2-3)--(m-1-1);\draw(m-2-3)--(m-1-3);\end{tikzpicture}$&
$\!\!\!\!\!\begin{tikzpicture}\matrix (m) [matrix of math nodes, row sep=17pt, column sep=-0pt]{
5&&&6\\
1&2&3&4\\ };
\draw(m-2-1)--(m-1-1);\draw(m-2-1)--(m-1-4); \draw(m-2-2)--(m-1-1);\draw(m-2-2)--(m-1-4); \draw(m-2-3)--(m-1-1);\draw(m-2-3)--(m-1-4); \draw(m-2-4)--(m-1-1);\draw(m-2-4)--(m-1-4); \end{tikzpicture}$&
$\!\!\!\!\!\begin{tikzpicture}\matrix (m) [matrix of math nodes, row sep=17pt, column sep=4pt]{
4&5&6\\
1&2&3\\ };
\draw(m-2-1)--(m-1-1); \draw(m-2-1)--(m-1-2); \draw(m-2-1)--(m-1-3);
\draw(m-2-2)--(m-1-1); \draw(m-2-2)--(m-1-2); \draw(m-2-2)--(m-1-3);
\draw(m-2-3)--(m-1-1); \draw(m-2-3)--(m-1-2); \draw(m-2-3)--(m-1-3); \end{tikzpicture}$&
$\!\!\!\!\!\begin{tikzpicture}\matrix (m) [matrix of math nodes, row sep=17pt, column sep=-0pt]{
 &6& &7&\\
1&2&3&4&5\\ };
\draw(m-2-1)--(m-1-2); \draw(m-2-2)--(m-1-2); \draw(m-2-3)--(m-1-2); \draw(m-2-4)--(m-1-2); \draw(m-2-5)--(m-1-2);
\draw(m-2-1)--(m-1-4); \draw(m-2-2)--(m-1-4); \draw(m-2-3)--(m-1-4); \draw(m-2-4)--(m-1-4); \draw(m-2-5)--(m-1-4); \end{tikzpicture}$&
$\!\!\!\!\!\begin{tikzpicture}\matrix (m) [matrix of math nodes, row sep=17pt, column sep=-0pt]{
5&6&7&\\
1&2&3&4\\ };
\draw(m-1-1)--(m-2-1); \draw(m-1-1)--(m-2-2); \draw(m-1-1)--(m-2-3); \draw(m-1-1)--(m-2-4);
\draw(m-1-2)--(m-2-1); \draw(m-1-2)--(m-2-2); \draw(m-1-2)--(m-2-3); \draw(m-1-2)--(m-2-4);
\draw(m-1-3)--(m-2-1); \draw(m-1-3)--(m-2-2); \draw(m-1-3)--(m-2-3); \draw(m-1-3)--(m-2-4); \end{tikzpicture}$\\[-1mm]
$\Z$& $\Z$& $\Z$& $\Z$& $\Z$& $\Z$\\
$\Z^4$& $\Z^5$& $\Z^6$& $\Z^6$& $\Z^7$& $\Z^7$\\
$\Z^6$& $\Z^{10}$& $\Z^{15}$& $\Z^{15}$& $\Z^{21}$& $\Z^{21}$\\
$\Z^4$& $\Z^{10}$& $\Z^{20}$& $\Z^{20}$& $\Z^{35}$& $\Z^{35}$\\
$\Z\!\oplus\!\Z_2$& $\Z^5\!\oplus\!\Z_2^3$& $\Z^{15}\!\oplus\!\Z_2^6$& $\Z^{15}\!\oplus\!\Z_2^9$& $\Z^{35}\!\oplus\!\Z_2^{10}$& $\Z^{35}\!\oplus\!\Z_2^{18}$\\
$\Z_2^3$& $\Z\!\oplus\!\Z_2^{12}$& $\Z^6\!\oplus\!\Z_2^{30}$& $\Z^6\!\oplus\!\Z_2^{45}$& $\Z^{21}\!\oplus\!\Z_2^{60}$& $\Z^{21}\!\oplus\!\Z_2^{108}$\\
$\Z_2^3$& $\Z_2^{18}$& $\Z\!\oplus\!\Z_2^{60}$& $\Z\!\oplus\!\Z_2^{96}$& $\Z^7\!\oplus\!\Z_2^{150}$& $\Z^{7}\!\oplus\!\Z_2^{294}$\\
$\Z_2$& $\Z_2^{12}$& $\Z_2^{60}$& $\Z_2^{120}$& $\Z\!\oplus\!\Z_2^{200}$& $\Z\!\oplus\!\Z_2^{504}$\\
$0$&$\Z_2^3$& $\Z_2^{31}$& $\Z_2^{105}$& $\Z_2^{155}$& $\Z_2^{651}$\\
&$0$& $\Z_2^{11}$& $\Z_2^{69}\!\oplus\!\Z_3$& $\Z_2^{90}$& $\Z_2^{714}\!\!\oplus\!\Z_3^4$\\
&$0$& $\Z_2^{10}$& $\Z_2^{30}\!\oplus\!\Z_3^5$& $\Z_2^{85}$& $\Z_2^{693}\!\!\oplus\!\Z_3^{24}$\\
&$0$& $\Z_2^{10}$& $\Z_2^6\!\oplus\!\Z_3^{10}$& $\Z_2^{100}$& $\Z_2^{564}\!\!\oplus\!\Z_3^{60}$\\
&&$\Z_2^5$& $\Z_3^{10}$& $\Z_2^{75}$& $\Z_2^{339}\!\!\oplus\!\Z_3^{80}$\\
&&$\Z_2$& $\Z_3^5$& $\Z_2^{30}$& $\Z_2^{126}\!\!\oplus\!\Z_3^{60}$\\
&&$0$&$\Z_3$& $\Z_2^5$& $\Z_2^{21}\!\!\oplus\!\Z_3^{24}$\\
&&&$0$&$0$& $\Z_3^4$\\
&&&&$0$&$0$\\
&&&&$0$&$0$\\
&&&&&$0$\\
&&&&&$0$\\
\end{longtable}}\vspace{-6pt}

\vspace{3pt}{\small\label{3.tableDiamond}
\begin{longtable}[c]{@{\hspace{0pt}}l@{\hspace{5pt}} l@{\hspace{5pt}} l@{\hspace{5pt}} l@{\hspace{5pt}} l@{\hspace{5pt}} l@{\hspace{5pt}} l@{\hspace{5pt}}}
$\!\!\!\!\!\begin{tikzpicture}\matrix (m) [matrix of math nodes, row sep=10pt, column sep=2pt]{
 &4&\\
2& &3\\
 &1&\\ };
\draw(m-3-2)--(m-2-1); \draw(m-3-2)--(m-2-3);
\draw(m-2-1)--(m-1-2); \draw(m-2-3)--(m-1-2);\end{tikzpicture}$\!\!\!&
$\begin{tikzpicture}\matrix (m) [matrix of math nodes, row sep=10pt, column sep=6pt]{
 &5&\\
2&3&4\\
 &1&\\ };
\draw(m-3-2)--(m-2-1); \draw(m-3-2)--(m-2-2); \draw(m-3-2)--(m-2-3);
\draw(m-2-1)--(m-1-2); \draw(m-2-2)--(m-1-2); \draw(m-2-3)--(m-1-2);\end{tikzpicture}$&
$\begin{tikzpicture}\matrix (m) [matrix of math nodes, row sep=10pt, column sep=-2pt]{
 &&6&\\
2&3&&4&5\\
 &&1&&\\ };
\draw(m-3-3)--(m-2-1); \draw(m-3-3)--(m-2-2); \draw(m-3-3)--(m-2-4); \draw(m-3-3)--(m-2-5);
\draw(m-2-1)--(m-1-3); \draw(m-2-2)--(m-1-3); \draw(m-2-4)--(m-1-3); \draw(m-2-5)--(m-1-3);\end{tikzpicture}$&
$\begin{tikzpicture}\matrix (m) [matrix of math nodes, row sep=10pt, column sep=-2pt]{
 &&7&\\
2&3&4&5&6\\
 &&1&&\\ };
\draw(m-3-3)--(m-2-1); \draw(m-3-3)--(m-2-2); \draw(m-3-3)--(m-2-3); \draw(m-3-3)--(m-2-4); \draw(m-3-3)--(m-2-5);
\draw(m-2-1)--(m-1-3); \draw(m-2-2)--(m-1-3); \draw(m-2-3)--(m-1-3); \draw(m-2-4)--(m-1-3); \draw(m-2-5)--(m-1-3);\end{tikzpicture}$&
$\begin{tikzpicture}\matrix (m) [matrix of math nodes, row sep=10pt, column sep=-2pt]{
 &&&8\\
2&3&4&&5&6&7\\
 &&&1\\ };
\draw(m-3-4)--(m-2-1); \draw(m-3-4)--(m-2-2); \draw(m-3-4)--(m-2-3); \draw(m-3-4)--(m-2-5); \draw(m-3-4)--(m-2-6); \draw(m-3-4)--(m-2-7);
\draw(m-2-1)--(m-1-4); \draw(m-2-2)--(m-1-4); \draw(m-2-3)--(m-1-4); \draw(m-2-5)--(m-1-4); \draw(m-2-6)--(m-1-4); \draw(m-2-7)--(m-1-4);\end{tikzpicture}$\\[-1mm]
$\Z$& $\Z$& $\Z$& $\Z$& $\Z$\\
$\Z^4$& $\Z^5$& $\Z^6$& $\Z^7$& $\Z^8$\\
$\Z^6$& $\Z^{10}$& $\Z^{15}$& $\Z^{21}$& $\Z^{28}$\\
$\Z^4\!\oplus\!\Z_2$& $\Z^{10}\!\oplus\!\Z_2$& $\Z^{20}\!\oplus\!\Z_2$& $\Z^{35}\!\oplus\!\Z_2$& $\Z^{56}\!\oplus\!\Z_2$\\
$\Z^1\!\oplus\!\Z_2^3$& $\Z^5\!\oplus\!\Z_2^5$& $\Z^{15}\!\oplus\!\Z_2^8$& $\Z^{35}\!\oplus\!\Z_2^{12}$& $\Z^{70}\!\oplus\!\Z_2^{17}$\\
$\Z_2^3\!\oplus\!\Z_3$& $\Z\!\oplus\!\Z_2^{10}\!\oplus\!\Z_3^2$& $\Z^6\!\oplus\!\Z_2^{25}\!\oplus\!\Z_3^2$& $\Z^{21}\!\oplus\!\Z_2^{51}\!\oplus\!\Z_3$& $\Z^{56}\!\oplus\!\Z_2^{91}\!\oplus\!\Z_3$\\
$\Z_2\!\oplus\!\Z_3^3$& $\Z_2^{10}\!\oplus\!\Z_3^8$& $\Z\!\oplus\!\Z_2^{40}\!\oplus\!\Z_3^{10}$& $\Z^7\!\oplus\!\Z_2^{110}\!\!\oplus\!\Z_3^7$& $\Z^{28}\!\oplus\!\Z_2^{245}\!\!\oplus\!\Z_3^{13}$\\
$\Z_3^3$& $\Z_2^5\!\oplus\!\Z_4\!\oplus\!\Z_3^{12}$& $\Z_2^{35}\!\oplus\!\Z_4^3\!\oplus\!\Z_3^{20}$& $\Z\!\oplus\!\Z_2^{136}\!\!\oplus\!\Z_4^5\!\oplus\!\Z_3^{21}$& $\Z^8\!\oplus\!\Z_2^{390}\!\!\oplus\!\Z_4^5\!\oplus\!\Z_3^{63}$\\
$\Z_3$& $\Z_2\!\oplus\!\Z_4^4\!\oplus\!\Z_3^8$& $\Z_2^{16}\!\oplus\!\Z_4^{15}\!\oplus\!\Z_3^{20}$& $\Z_2^{103}\!\!\oplus\!\Z_4^{30}\!\oplus\!\Z_3^{35}$& $\Z\!\oplus\!\Z_2^{411}\!\!\oplus\!\Z_4^{35}\!\oplus\!\Z_3^{161}$\\
$0$& $\Z_4^6\!\oplus\!\Z_3^2$& $\Z_2^3\!\oplus\!\Z_4^{30}\!\oplus\!\Z_3^{10}\!\oplus\!\Z_5$& $\Z_2^{58}\!\oplus\!\Z_4^{75}\!\oplus\!\Z_3^{35}\!\oplus\!\Z_5^4$& $\Z_2^{357}\!\!\oplus\!\Z_4^{105}\!\!\oplus\!\Z_3^{245}\!\!\oplus\!\Z_5^9$\\
& $\Z_4^4$&$\Z_4^{30}\!\oplus\!\Z_3^2\!\oplus\!\Z_5^5$& $\Z_2^{41}\!\oplus\!\Z_4^{100}\!\!\oplus\!\Z_3^{21}\!\oplus\!\Z_5^{24}$& $\Z_2^{351}\!\!\oplus\!\Z_4^{175}\!\!\oplus\!\Z_3^{231}\!\!\oplus\!\Z_5^{63}$\\
& $\Z_4$&$\Z_4^{15}\!\oplus\!\Z_5^{10}$& $\Z_2^{36}\!\oplus\!\Z_4^{75}\!\oplus\!\Z_3^8\!\oplus\!\Z_5^{60}$& $\Z_2^{365}\!\!\oplus\!\Z_4^{175}\!\!\oplus\!\Z_3^{138}\!\!\oplus\!\Z_5^{189}$\\
& $0$&$\Z_4^3\!\oplus\!\Z_5^{10}$& $\Z_2^{27}\!\oplus\!\Z_4^{30}\!\oplus\!\Z_3^7\!\oplus\!\Z_5^{80}$& $\Z_2^{315}\!\!\oplus\!\Z_4^{105}\!\!\oplus\!\Z_3^{78}\!\oplus\!\Z_5^{315}$\\
& &    $\Z_5^5$& $\Z_2^{22}\!\oplus\!\Z_4^5\!\oplus\!\Z_3^{15}\!\oplus\!\Z_5^{60}$& $\Z_2^{245}\!\!\oplus\!\Z_4^{35}\!\oplus\!\Z_3^{111}\!\!\oplus\!\Z_5^{315}\!\!\oplus\!\Z_7$\\
& &    $\Z_5$& $\Z_2^{21}\!\oplus\!\Z_3^{20}\!\oplus\!\Z_5^{24}$& $\Z_2^{215}\!\!\oplus\!\Z_4^{5}\!\oplus\!\Z_3^{175}\!\!\oplus\!\Z_5^{189}\!\!\oplus\!\Z_7^7$\\
& &    $0$&$\Z_2^{15}\!\oplus\!\Z_3^{15}\!\oplus\!\Z_5^4$& $\Z_2^{180}\!\!\oplus\!\Z_3^{175}\!\!\oplus\!\Z_5^{63}\!\oplus\!\Z_7^{21}$\\
& &       &$\Z_2^6\!\oplus\!\Z_3^6$& $\Z_2^{105}\!\!\oplus\!\Z_3^{105}\!\!\oplus\!\Z_5^{9}\!\oplus\!\Z_7^{35}$\\
& &       &$\Z_2\!\oplus\!\Z_3$& $\Z_2^{35}\!\oplus\!\Z_3^{35}\!\oplus\!\Z_7^{35}$\\
& &       &$0$& $\Z_2^{5}\!\oplus\!\Z_3^{5}\!\oplus\!\Z_7^{21}$\\
& &       &   & $\Z_7^7$\\
& &       &   & $\Z_7$\\
& &       &   & $0$\\
\end{longtable}}\vspace{2pt}

\vspace{4mm}
\section{Torsion properties}

\begin{Prp}\label{IntervalTorsion} For any $a,b\!\in\![n]$ and any $m\!\in\!\N$ such that the interval $[a,b]$ in $\preceq$ contains more than $m$ elements, there is a direct summand $\Z_m$ in $H_\ast(\frak{gl}_n^\preceq;\Z)$.
\end{Prp} Therefore the (co)homology of $\frak{gl}_n^\preceq$ contains $p$-torsion for all primes $p$ that are smaller than the size of the largest interval in the poset $\preceq$.
\begin{proof}
Given $a\!\prec\!x_1,\ldots,x_{m\!-\!1}\!\prec\!b$, we define the wedge $v\!=\!e_{ab}\bigwedge_{i=1}^{m\!-\!1}e_{ax_i}e_{x_ib}$, which has weights $w_v\!=\!(-m,0,\ldots,0,m)$. Thus $v\!\in\!\Ker\partial$ and $mv\!\in\!\Im\partial$, so it remains to show that $lv\!\in\!\Im\partial$ implies $l\!\in\!m\Z$. All the edges in digraph $\Gamma_{C_\star}$ that have $v$ as an endpoint are listed below (obtained by splitting $e_{ab}$ or some $e_{ax_i}$ or some $e_{x_ib}$), and we define $\mathcal{M}$ as the set of red edges: \vspace{-2mm}
$$\begin{tikzpicture}[baseline=(m.center)]\matrix (m) [matrix of math nodes, row sep=-3pt, column sep=30pt]{
e_{aa}v=e_{aa}e_{ab}\bigwedge_ie_{ax_i}e_{x_ib}                         &\\
e_{bb}v=e_{bb}e_{ab}\bigwedge_ie_{ax_i}e_{x_ib}                         &e_{ab}e_{ax}e_{xb}\bigwedge_{i\neq1}e_{ax_i}e_{x_ib}=:v_x\\
v'_x:=e_{ax}e_{xb}\bigwedge_ie_{ax_i}e_{x_ib}                           &e_{ab}\bigwedge_ie_{ax_i}e_{x_ib}=v\\
v'_{ry}:=e_{ab}e_{ay}e_{yx_r}e_{x_rb}\bigwedge_{i\neq r}e_{ax_i}e_{x_ib} &e_{ab}e_{ay}e_{yb}\bigwedge_{i\neq r}e_{ax_i}e_{x_ib}=:v_{ry}\\
v'_{sz}:=e_{ab}e_{ax_s}e_{x_sz}e_{zb}\bigwedge_{i\neq s}e_{ax_i}e_{x_ib} &e_{ab}e_{az}e_{zb}\bigwedge_{i\neq s}e_{ax_i}e_{x_ib}=:v_{sz}\\};
\draw[->](m-1-1.east)--(m-3-2.west); \draw[->](m-2-1.east)--(m-3-2.west); \draw[->](m-3-1.east)--(m-3-2.west); \draw[->](m-4-1.east)--(m-3-2.west); \draw[->](m-5-1.east)--(m-3-2.west);
\draw[->,red](m-3-1.east)--(m-2-2.west); \draw[->,red](m-4-1.east)--(m-4-2.west); \draw[->,red](m-5-1.east)--(m-5-2.west);
\end{tikzpicture}\vspace{-2mm}$$
Out of $v'_{ry}$ (resp. $v'_{sz}$) an edge goes only to $v$ and $v_{ry}$ (resp. $v_{sz}$). There is an edge from $v'_x$ to $v_{ry}$ (resp. $v_{sz}$) iff $x\!=\!y$ (resp. $x\!=\!z$). Notice that for any choices $a\!\prec\!x\!\prec\!b$, $a\!\prec\!y\!\prec\!x_r\!\prec\!b$, $a\!\prec\!x_s\!\prec\!z\!\prec\!b$ the vertices $v_{ry},v_{sz},v'_x,v'_{ry},v'_{sz}$ are pairwise distinct, but $v_x\!=\!v_{ry}$ (resp. $v_x\!=\!v_{sz}$) iff $1\!=\!r$ (resp. $1\!=\!s$) and $x\!=\!y$ (resp. $x\!=\!z$). For any $a\!\prec\!y\!\prec\!x_1$ (resp. $x_1\!\prec\!z\!\prec\!b$) we remove  $v'_y\!\to\!v'_{1y}$ (resp. $v'_z\!\to\!v'_{1z}$). Then \vspace{-1mm}
$$\mathcal{M}=\{v'_x\!\to\!v_x, v'_{ry}\!\to\!v_{ry}, v'_{sz}\!\to\!v_{sz};\, r\!\neq\!1\!\neq\!s\}\vspace{-1mm}$$
is a Morse matching, and from $v'_y$ (resp. $v'_z$) to $v$ there are $m$ zig-zag paths, all of the same sign, so \smash{$\mathring{\partial}(v'_x)=mv$} and the result follows. For various choices of $a,b\!\in\![n]$ and $m\!\in\!\N$ such that $|[a,b]|\!>\!m$, the union of resulting $\mathcal{M}$'s is a Morse matching on the whole chain complex, so we get distinct direct summands $\Z_m$ in $H_\ast(\frak{gl}_n^\preceq;\Z)$.
\end{proof}\vspace{1mm}

A poset is \emph{bounded} when it contains a largest and a smallest element.
\begin{Crl} If $\preceq$ is bounded, then $H_\ast(\frak{gl}_n^\preceq;\Z)$ contains $p$-torsion iff $p\!<\!n$.
\end{Crl}
\begin{proof} If $a$ is the largest element and $b$ the smallest element, then interval $[a,b]$ has $n$ elements and by \ref{IntervalTorsion} $H_\ast(\frak{gl}_n^\preceq;\Z)$ contains $p$-torsion for every $p\!<\!n$. Conversely, for any prime $p\!\geq\!n$ all nonzero weights of indices are smaller than $p$, so the $p$-complex is empty, hence $H_\ast(\frak{gl}_n^\preceq;\Z)$ does not contain $p$-torsion.
\end{proof}
Thus for bounded posets (and we suspect for all posets), $H_\ast(\frak{gl}_n^\preceq;\Z)$ is \emph{torsion-convex}. This is not true for solvable Lie algebras which do not come from posets, e.g. the nilpotent example $L_{6,19}(p)$ from \cite[p.647]{citeGraafC6NLAFC2} contains only $p$-torsion. \vspace{1mm}

\begin{Rmk} If poset $\preceq$ is not bounded, i.e. contains several minimal and/or several maximal elements, then the largest torsion can be much bigger than the largest interval. For example, if the Hasse diagram of the poset is the $p$-complete bipartite graph (it has $2p$ vertices and height $1$), then the largest interval has $2$ elements but $H_\ast(\frak{gl}_n^\preceq;\Z)$ contains $p$-torsion, generated by the wedge of all nondiagonals. Indeed, for $v\!=\!\bigwedge_{i\prec j}e_{ij}$ we have $\partial(v)\!=\!0$ and $\partial(e_{ii}v)=\pm p\,v$ for any $i$. For example, in \vspace{-3mm}
$$\begin{tikzpicture}[baseline=(m.center)]\matrix (m) [matrix of math nodes, row sep=20pt, column sep=15pt]{
3&4\\
1&2\\ };
\draw(m-1-1)--(m-2-1); \draw(m-1-1)--(m-2-2);
\draw(m-1-2)--(m-2-1); \draw(m-1-2)--(m-2-2);\end{tikzpicture}\hspace{10mm}
\begin{tikzpicture}[baseline=(m.center)]\matrix (m) [matrix of math nodes, row sep=20pt, column sep=15pt]{
4&5&6\\
1&2&3\\ };
\draw(m-1-1)--(m-2-1); \draw(m-1-1)--(m-2-2); \draw(m-1-1)--(m-2-3);
\draw(m-1-2)--(m-2-1); \draw(m-1-2)--(m-2-2); \draw(m-1-2)--(m-2-3);
\draw(m-1-3)--(m-2-1); \draw(m-1-3)--(m-2-2); \draw(m-1-3)--(m-2-3);\end{tikzpicture}\hspace{10mm}
\begin{tikzpicture}[baseline=(m.center)]\matrix (m) [matrix of math nodes, row sep=20pt, column sep=15pt]{
6&7&8&9&10\\
1&2&3&4&5\\ };
\draw(m-1-1)--(m-2-1); \draw(m-1-1)--(m-2-2); \draw(m-1-1)--(m-2-3); \draw(m-1-1)--(m-2-4); \draw(m-1-1)--(m-2-5);
\draw(m-1-2)--(m-2-1); \draw(m-1-2)--(m-2-2); \draw(m-1-2)--(m-2-3); \draw(m-1-2)--(m-2-4); \draw(m-1-2)--(m-2-5);
\draw(m-1-3)--(m-2-1); \draw(m-1-3)--(m-2-2); \draw(m-1-3)--(m-2-3); \draw(m-1-3)--(m-2-4); \draw(m-1-3)--(m-2-5);
\draw(m-1-4)--(m-2-1); \draw(m-1-4)--(m-2-2); \draw(m-1-4)--(m-2-3); \draw(m-1-4)--(m-2-4); \draw(m-1-4)--(m-2-5);
\draw(m-1-5)--(m-2-1); \draw(m-1-5)--(m-2-2); \draw(m-1-5)--(m-2-3); \draw(m-1-5)--(m-2-4); \draw(m-1-5)--(m-2-5);\end{tikzpicture}\vspace{-3mm}$$
the first case induces $2$-torsion, second induces $2,3$-torsion, third induces $2,3,5$-torsion. This shows the difference between $H_\ast(\frak{gl}_n^\preceq;\Z)$ and $H_\ast(\frak{gl}_n^\prec;\Z)$: the former can have a lot of torsion even if all intervals are small, but the latter in the case of bipartite Hasse diagrams is torsion-free (since differentials are zero).
\end{Rmk}\vspace{1mm}

\begin{Crl} If $H_\ast(\frak{gl}_n^\preceq;\Z)$ does not contain $2$-torsion, then its Hasse diagram is a forest of height $\leq1$, hence $H_\ast(\frak{gl}_n^\preceq;\Z)$ is torsion-free.
\end{Crl}
\begin{proof} If $H_\ast(\frak{gl}_n^\preceq;\Z)$ contains no $2$-torsion, then by \ref{IntervalTorsion} the poset does not contain intervals with $3$ or more elements, so its height is at most $1$. If $\overline{\Gamma}_\preceq$ contained a cycle\vspace{-3mm}
$$\begin{tikzpicture}[baseline=(m.center)]\matrix (m) [matrix of math nodes, row sep=20pt, column sep=15pt]{
y_1&y_2&y_3&\ldots&y_k\\
x_1&x_2&x_3&\ldots&x_k,\\ };
\draw(m-1-1)--(m-2-1); \draw(m-1-2)--(m-2-2); \draw(m-1-3)--(m-2-3); \draw(m-1-5)--(m-2-5);
\draw(m-1-1)--(m-2-2); \draw(m-1-2)--(m-2-3); \draw(m-1-3)--(m-2-4); \draw(m-1-4)--(m-2-5); \draw(m-1-5)--(m-2-1);
\end{tikzpicture}\vspace{-3mm}$$
then wedge $\bigwedge_{i\in\Z_k}\!e_{x_iy_i}e_{x_iy_{i\!-\!1}}\in C_{\star,2}(\frak{gl}^\preceq)$ would generate $2$-torsion, a contradiction. Therefore $\Gamma_\preceq$ is a forest. Hence every wedge of $e_{xy}$'s (viewed as a subgraph of $\overline{\Gamma}_\preceq$) is a forest, so it has leaves (indices of weight $\pm1$). Thus every $C_{\star,p}(\frak{gl}^\preceq)$ is empty.
\end{proof}\vspace{1mm}

\begin{Cnj} If $H_\ast(\frak{gl}_n^\preceq;\Z)$ contains $p$-torsion, then it contains $p'$-torsion for every prime $p'\!<\!p$.
\end{Cnj}\vspace{1mm}

\begin{Prp} If $C_{\star,p}(\frak{gl}_n^\prec)$ is nonempty, then so is $C_{\star,p'}(\frak{gl}_n^\prec)$ for every $p'\!<\!p$.
\end{Prp}
\begin{proof} By assumption, there exists a wedge $\emptyset\!\neq\!v\!\in\!C_{\star}(\frak{gl}_n^\preceq)$ with weights $w_v\!\in\!(p\Z)^n$. Let $\Gamma_v$ be a digraph, whose edges $i\!\to\!j$ are the elements $e_{ij}\!\in\!v$. The indices appearing in $v$ are partitioned into three sets: those with positive/negative/zero weights. Let $\Gamma_{(v)}$ be the bipartite multigraph with vertices the indices in $v$ with positive/negative weights, denoted $a_i$ and $b_j$. Given indices $a$ and $b$ with $w_a\!<\!0\!<\!w_b$, and given a directed path $\gamma$ from $a$ to $b$ in $\Gamma_v$, we remove ell edges of $\gamma$ from $\Gamma_v$ and add an edge $a\!\to\!b$ to $\Gamma_{(v)}$. We keep doing this until $\Gamma_v$ has no edges. For example, if the Hasse diagram is below left and $v\!\in\!C_{\star,3}$ consists of edges pictured below centre, then the corresponding $\Gamma_v$ is below right.\vspace{-2mm}
$$\begin{tikzpicture}[baseline=(m.center)]\matrix (m) [matrix of math nodes, row sep=8pt, column sep=8pt]{
6 &7 &8 &9 &10\\
&&3&&\\
1 &2 & &4 &5\\ };
\draw(m-3-1)--(m-1-1); \draw(m-3-1)--(m-1-2); \draw(m-3-1)--(m-1-3); \draw(m-3-5)--(m-1-3); \draw(m-3-5)--(m-1-4); \draw(m-3-5)--(m-1-5);
\draw(m-3-2)--(m-2-3); \draw(m-3-4)--(m-2-3); \draw(m-2-3)--(m-1-1); \draw(m-2-3)--(m-1-2); \draw(m-2-3)--(m-1-3); \draw(m-2-3)--(m-1-4); \draw(m-2-3)--(m-1-5);
\end{tikzpicture}\hspace{8mm}
\begin{tikzpicture}[baseline=(m.center)]\matrix (m) [matrix of math nodes, row sep=5pt, column sep=7pt]{
6 &7 &8 &9 &10\\
&&3&&\\
1 &2 & &4 &5\\ };
\draw(m-3-1)--(m-1-1); \draw(m-3-1)--(m-1-2); \draw(m-3-1)--(m-1-3); \draw(m-3-2)--(m-1-1); \draw(m-3-2)--(m-1-2); \draw(m-3-2)--(m-2-3);
\draw(m-3-5)--(m-1-3); \draw(m-3-5)--(m-1-4); \draw(m-3-5)--(m-1-5); \draw(m-3-4)--(m-2-3); \draw(m-3-4)--(m-1-4); \draw(m-3-4)--(m-1-5);
\draw(m-3-2)--(m-2-3); \draw(m-3-4)--(m-2-3); \draw(m-2-3)--(m-1-1); \draw(m-2-3)--(m-1-2); \draw(m-2-3)--(m-1-3); \draw(m-2-3)--(m-1-4); \draw(m-2-3)--(m-1-5);
\end{tikzpicture}\hspace{8mm}
\begin{tikzpicture}[baseline=(m.center)]\matrix (m) [matrix of math nodes, row sep=5pt, column sep=7pt]{
b_1 &b_2 &b_3 &b_4 &b_5\\
&&a_3&&\\
a_1 &a_2 & &a_4 &a_5\\ };
\draw(m-3-1)--(m-1-1); \draw(m-3-1)--(m-1-2); \draw(m-3-1)--(m-1-3); \draw(m-3-2)--(m-1-1); \draw(m-3-2)--(m-1-2); \draw[bend left](m-3-2)--(m-1-1);
\draw(m-3-5)--(m-1-3); \draw(m-3-5)--(m-1-4); \draw(m-3-5)--(m-1-5); \draw[bend left](m-3-4)--(m-1-5); \draw(m-3-4)--(m-1-4); \draw(m-3-4)--(m-1-5);
\draw(m-3-2)--(m-2-3); \draw(m-3-4)--(m-2-3); \draw(m-2-3)--(m-1-1); \draw(m-2-3)--(m-1-2); \draw(m-2-3)--(m-1-3); \draw(m-2-3)--(m-1-4); \draw(m-2-3)--(m-1-5);
\end{tikzpicture}\hspace{8mm}
\vspace{-2mm}$$ Our $\Gamma_{(v)}$ is a bipartite multigraph whose vertices have degrees in $p\N$. We wish to obtain a graph whose vertices all have degree $p$. We create a new multigraph $\Gamma'_{(v)}$: instead of every vertex $a$ of degree $pk$, we draw $k$ copies of it, each of which has $p$ edges to those $b$'s to which $a$ was connected. We create a new multigraph $\Gamma''_{(v)}$: instead of every vertex $b$ of degree $pl$, we draw $l$ copies of it, each of which has $p$ edges to those $a$'s to which $b$ was connected. Now $\Gamma_v''$ is a $p$-regular bipartite multigraph. By the K\"{o}nig-Egerv\'{a}ry theorem, the edges of $\Gamma''_{(v)}$ can be colored with only $p$ colors. Let $\Gamma''$ be the subgraph consisting of those edges colored by the first $p'$ colors, so it is $p'$-regular. Now we reverse the process: we join $b$'s that came from a $b$ of degree $pl$ into a single vertex of degree $p'l$, and then join $a$'s that came from a $a$ of degree $pk$ into a single vertex of degree $p'k$. The result $\Gamma$ is a subgraph of $\Gamma_v$, so it corresponds to a subwedge of $v$ which lies in $C_{\star,p'}$.
\end{proof}

\begin{Rmk} Chain complex $C_{\star,p}$ can be nonempty but still contractible. E.g. for \vspace{-2mm}
$$\begin{tikzpicture}[baseline=(m.center)]\matrix (m) [matrix of math nodes, row sep=5pt, column sep=13pt]{
b_1 &b_2 & b_3\\
x&&\\
a_1 &a_2 & a_3\\ };
\draw(m-3-1)--(m-2-1); \draw(m-2-1)--(m-1-1); \draw(m-3-1)--(m-1-2); \draw(m-3-1)--(m-1-3);
\draw(m-3-2)--(m-1-1); \draw(m-3-2)--(m-1-2); \draw(m-3-2)--(m-1-3);
\draw(m-3-3)--(m-1-1); \draw(m-3-3)--(m-1-2); \draw(m-3-3)--(m-1-3); \end{tikzpicture}\vspace{-2mm}$$
our $C_{\star,3}$ has a basis consisting of $u\!=\!e_{a_1x}e_{xb_1}\bigwedge_{(i,j)\neq(1,1)}e_{a_ib_j}$ and $v\!=\!\bigwedge_{(i,j)}e_{a_ib_j}$ and $\emptyset$, with $\partial(u)=\pm v$ and $\partial(v)\!=\!0$, so $C_{\star,3}$ is contractible. In fact, if the poset consists of elements $a_1,\ldots,a_p,b_1,\ldots,b_p,x_1,\ldots,x_n$ with $n\!<\!p$ and relations $a_i\!\prec\!b_j$ and $a_i\!\prec\!x_i\!\prec\!b_i$ for all $i$ and $j$, then $C_{\star,p}$ is isomorphic to the simplicial chain complex of the $n$-cube, and is therefore homotopy equivalent to $R\!\leftarrow\!0\!\leftarrow\!0\!\leftarrow\!0\!\leftarrow\!\ldots$.
\end{Rmk}

\begin{Prp}\label{niltorsion} $H_\ast(\frak{nil}_n;\Z)$ contains a direct summand $\Z_{n\!-\!2}$.
\end{Prp}
Determining what torsion appears in the general $H_\ast(\frak{gl}_n^\prec;\Z)$ is more difficult.
Thus tables $\big(H_\ast(\frak{nil}_n;\Z)\big)_{\!n\in\N}$ and $\big(H_\ast(\frak{sol}_n;\Z)\big)_{\!n\in\N}$ contain all prime powers $\Z_{p^r}$.
\begin{proof} Let $C_{[w]}$ be the direct summand subcomplex of $C_\star(\frak{nil}_n;\Z)$, spanned by all wedges with weight vector $w\!=\!(-1,3\!-\!n,1,\ldots,1)\!\in\!\Z^n$. For any  $v\!\in\!C_{[w]}$ there holds:\, $v$ either contains all $e_{23},e_{24},\ldots,e_{2n}$ (iff $v\!=\!e_{12}\bigwedge_{3\leq i\leq n}\!e_{2i}\!=:\!\alpha$) or exactly one of them is missing;\, $v$ contains exactly one $e_{1x}$ and one $e_{yn}$ (which may be equal);\, $e_{1n}\!\in\!v$ iff $v\!=\!e_{1n}\bigwedge_{3\leq i<n}\!e_{2i}\!=:\!\beta$\!. If $e_{2n}\!\in\!v\!\neq\!\alpha$\!, then $e_{2i}\!\notin\!v$ for a unique $i\!<\!n$ and $e_{in}\!\notin\!v$. If $e_{2n}\!\notin\!v\!\neq\!\beta$, then $\bigwedge_{3\leq i<n}\!e_{2i}\!\subseteq\!v$ and $e_{in}\!\in\!v$ for a unique $i\!>\!2$. Thus \vspace{-1mm}
$$\mathcal{M}\!=\!\big\{e_{2i}e_{in}\ldots \!\longrightarrow\! e_{2n}\ldots \big\}\vspace{-2mm}$$ is a matching. There does not exist a pair of distinct edges
\raisebox{7pt}{$\xymatrix@R=-3pt@C=8mm{&\scriptscriptstyle e_{2n}\ldots\\
\scriptscriptstyle e_{2i}e_{in}\ldots\ar[ru]\ar[r]   &\scriptscriptstyle e_{2n}\ldots\\}$}
(because $e_{in}$ is unique in a wedge, so we must bracket it with $e_{2i}$ to obtain $e_{2n}$), hence $\mathcal{M}$ is a Morse matching and every zig-zag path contains at most one zig-zag. The critical vertices are $\mathring{\mathcal{M}}\!=\! \{\alpha,\beta\}$. From $\alpha$ to $\beta$ there is one direct edge $\alpha\!\overset{\scriptscriptstyle(-1)^n}{\longrightarrow}\!\beta$ and $n\!-\!3$ zig-zag paths $\alpha\!\overset{\scriptscriptstyle(-1)^r}{\longrightarrow}\! e_{1r}\bigwedge_{i\neq r}\!e_{2i}\!\overset{\scriptscriptstyle(-1)^{n+r}}{\longleftarrow}\! e_{1r}e_{rn}\bigwedge_{i\neq n}\!e_{2i} \!\overset{\scriptscriptstyle-1}{\longrightarrow}\!\beta$ which have sign $(-1)^{r+n+r+1-1}\!=\!(-1)^n$\!, hence $\mathring{\partial}(\alpha)=\pm(n\!-\!2)\beta$ and we are finished.
\end{proof}
Notice that \ref{subsetopposite} and \ref{niltorsion} confirm the conjecture \cite[1.16.(1), p.141]{citeJollenbeckADMTACA}.


\vspace{4mm}
\section{Posets of height 1}\label{HGT1} In this section, all posets are assumed to have height $1$, so Hasse diagrams correspond to bipartite graphs. A wedge in the $p$-complex $C_{\star,p}(\frak{gl}_n^\prec)$ in this case corresponds to a subset of edges of the Hasse diagram $\Gamma_{\!\preceq}$, so that those edges constitute a subgraph of $\overline{\Gamma}_{\!\preceq}$ in which all vertices have degrees in $p\Z$.
\par This notion is similar to $p$-regular subgraphs (which are those in which every vertex has degree $p$), so we shall call such full subgraphs \emph{$p^+$\!-regular}.
\par Height $1$ also means that $\frak{gl}_n^\prec$ has all brackets zero, so we just need to count all $p^+$\!-regular subgraphs with $k$ edges to obtain $H_k(\frak{gl}_n^\preceq;\Z_p)$. This can still be quite complex, since subgraphs of different sizes contribute to different degrees. \vspace{1mm}

\begin{Crl} $\dim H_k(\frak{gl}_n^\preceq;\Z_2)= \sum_{i}\!\binom{n}{k-i}|\{\text{Eulerian subgraphs in }\overline{\Gamma}_{\!\preceq}\text{ of size }i\}|$.
\end{Crl}
This is because $2^+$\!-regular graphs are those in which every vertex has even degree, i.e. Eulerian graphs. We look at a few particular cases.\vspace{2mm}

\subsection{Path and cycle posets} Let our poset $\preceq$ be given by the Hasse diagram below which is an $n$-cycle. Thus $\dim \frak{gl}^\preceq_{2n}=4n$.\vspace{-2mm}
$$\begin{tikzpicture}[baseline=(m.center)]\matrix (m) [matrix of math nodes, row sep=20pt, column sep=13pt]{
b_1 &b_2 & \ldots & b_n\\
a_1 &a_2 & \ldots & a_n\\ };
\draw(m-2-1)--(m-1-1); \draw(m-2-1)--(m-1-4);
\draw(m-2-2)--(m-1-2); \draw(m-2-2)--(m-1-1);
\draw(m-2-3)--(m-1-3); \draw(m-2-3)--(m-1-2);
\draw(m-2-4)--(m-1-4); \draw(m-2-4)--(m-1-3); \end{tikzpicture}\hspace{2mm}\frak{gl}^\preceq=\bigg\{ \raisebox{-6mm}{\includegraphics[width=0.11\textwidth]{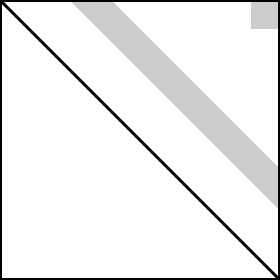}}\bigg\}\vspace{-2mm}$$

\begin{Prp} \!$\HP_{\!C\!_\star\!(\frak{gl}^\preceq;\Z_2)} \!\!=\! (1\!+\!t)^{2n}(1\!+\!t^{2n})$ and $\HP_{\!C\!_\star\!(\frak{gl}^\preceq;\Z_p)} \!\!=\! (1\!+\!t)^{2n}$ \!if $p\!>\!2$.
\end{Prp}
\begin{proof} In the given Hasse diagram, there are only two $p^+$\!-regular subgraphs: the empty subgraph (wedge $\emptyset$) and the whole subgraph (wedge $\bigwedge_{i\in\Z_n}\!e_{a_ib_i}e_{a_ib_{i\!-\!1}}$).
\end{proof}
\vspace{2mm}

The case when $\Gamma_{\!\preceq}$ is a path, or more generally a tree (of height $1$), is trivial:
\begin{Prp} If $\Gamma_{\!\preceq}$ is a tree, then $H_k(\frak{gl}^\preceq_n;\Z)\cong \Z^{\binom{n}{k}}$.
\end{Prp}
\begin{proof} For any $\preceq$ we have $H_k(\frak{gl}^\preceq_n;\Q)\cong \Q^{\binom{n}{k}}$. In the case of a tree, any subgraph of $\overline{\Gamma}_{\!\preceq}$ is a forest, hence for any prime $p$ there are no $p^+$\!-regular subgraphs. Thus $H_k(\frak{gl}^\preceq_n;\Z)$ does not contain any torsion, and the result follows.
\end{proof}
\vspace{2mm}

\subsection{Complete bipartite posets} Let our poset $\preceq$ be given by the Hasse diagram below which is a complete bipartite graph. Thus $\dim \frak{gl}^\preceq_{m+n}=m\!+\!n\!+\!mn$.\vspace{-2mm}
$$\begin{tikzpicture}[baseline=(m.center)]\matrix (m) [matrix of math nodes, row sep=20pt, column sep=13pt]{
b_1 &b_2 & \ldots & b_n\\
a_1 &a_2 & \ldots & a_m&\\ };
\draw(m-2-1)--(m-1-1); \draw(m-2-1)--(m-1-2); \draw(m-2-1)--(m-1-3); \draw(m-2-1)--(m-1-4);
\draw(m-2-2)--(m-1-1); \draw(m-2-2)--(m-1-2); \draw(m-2-2)--(m-1-3); \draw(m-2-2)--(m-1-4);
\draw(m-2-3)--(m-1-1); \draw(m-2-3)--(m-1-2); \draw(m-2-3)--(m-1-3); \draw(m-2-3)--(m-1-4);
\draw(m-2-4)--(m-1-1); \draw(m-2-4)--(m-1-2); \draw(m-2-4)--(m-1-3); \draw(m-2-4)--(m-1-4); \end{tikzpicture}\hspace{2mm} \frak{gl}^\preceq=\bigg\{ \raisebox{-6mm}{\includegraphics[width=0.11\textwidth]{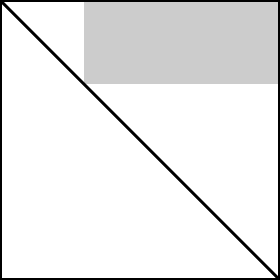}}\bigg\}\vspace{-2mm}$$

\begin{Prp} $\dim H_k(C_{\star,p}(\frak{gl}_n^\prec))$ equals the number of all $m\!\times\!n$ matrices with $k$ entries $1$ and $mn\!-\!k$ entries $0$, such that the sum of every row/column is in $p\Z$.
\end{Prp}
\begin{proof} The edges in $\overline{\Gamma}_{\!\preceq}$ correspond to elements of $[m]\!\times\![n]$, hence the wedges of $\overline{\Gamma}_{\preceq}$ correspond to subsets of $[m]\!\times\![n]$, or equivalently, to $0/1$-matrices of size $m\!\times\!n$. The $p^+$\!-regular subgraphs of $\overline{\Gamma}_{\preceq}$ correspond to those $0/1$-matrices in which the sum of every column and the sum of every row is a multiple of $p$.
\end{proof}\vspace{2mm}

Let $\HP_{m,n,p}$ denote the generating function for the homology of our $\frak{gl}_{m+n}^\preceq(\Z_p)$. Thus $\HP_{m,n,p} \!=\! \HP_{n,m,p}$ and $\HP_{1,n,p} \!=\! (1\!+\!t)^{n+1}$ and $\HP_{m,n,p} \!=\! (1\!+\!t)^{m+n}$ for $p\!>\!\min\{m,n\}$. Let $\epsilon=e^{2\pi\mathbf{i}/p}\!\in\!\C$ be the $p$-th primitive root of unity.
\begin{Prp}\label{pnp} $\HP_{p,n,p}=(1\!+\!t)^{p+n}\sum_{i\in p\N}\!\binom{n}{i}t^{pi}=\frac{(1+t)^{p+n}}{p}\sum_{j\in[p]}(1\!+\!\epsilon^j\,t^p)^n$.
\end{Prp}
\begin{proof} If $m\!=\!p$, then in a $p^+$\!-regular subgraph, every $b_i$ has degree $p$ or $0$. Thus obtaining a $p^+$\!-regular subgraph means that out of $b_1,\ldots,b_n$ we must choose $i\!\in\!p\N$ elements. Such a subgraph has $pi$ edges. The second equality follows from \ref{DFT}.
\end{proof}\vspace{2mm}

\begin{Prp}\label{Stanley} \!\!$\frac{(1+t)^{m+n}}{2^{m+n}}\sum_{i,j}\!\binom{m}{i}\binom{n}{j}(1\!+\!t)^{(m\!-\!i)(n\!-\!j)+ij}(1\!-\!t)^{(m\!-\!i)j+(n\!-\!j)i}$ equals $\HP_{m,n,2}$, so $\dim H_k(\frak{gl}_{m+n}^\preceq;\Z_2)=\sum_{i,j,l}\!\frac{\binom{m}{i}\binom{n}{j}\binom{(m\!-\!i)(n\!-\!j)+ij+m+n}{l}\binom{(m\!-\!i)j+(n\!-\!j)i}{k-l}(-\!1)^{k-l}}{2^{m+n}}$.
\end{Prp} \vspace{-0mm} This formula and its proof are due to Richard Stanley (personal communication).\vspace{-0mm}
\begin{proof} We call a $0/1$-matrix $E$ \emph{even} when the sum of every column and row is even. 
\begin{longtable}[c]{l}
$\HP_{m,n,2}=(1\!+\!t)^{m+n}\sum_{E\in\{0,1\}^{m\!\times\!n}\!\text{ even}}t^{\sum_{ij}\!E_{ij}}$\\
$\overset{(1)}{=} \frac{(1+t)^{m+n}}{2^{m+n}}\! \sum_{E\in\{0,1\}^{m\!\times\!n}}t^{\sum_{ij}\!E_{ij}} \prod_i\!\big(1\!+\!(-1)^{\sum_j\!E_{ij}}\big)\prod_j\!\big(1\!+\!(-1)^{\sum_i\!E_{ij}}\big)$\\
$\overset{(2)}{=} \frac{(1+t)^{m+n}}{2^{m+n}}\! \sum_{E\in\{0,1\}^{m\!\times\!n}}t^{\sum_{ij}\!E_{ij}} \Big(\!\sum_{x\in\{-1,1\}^m}\!\prod_i\!x_i^{\sum_j\!E_{ij}}\!\Big)\Big(\!\sum_{y\in\{-1,1\}^n}\!\prod_j\!y_j^{\sum_i\!E_{ij}}\!\Big)$\\
$\overset{(3)}{=} \frac{(1+t)^{m+n}}{2^{m+n}}\! \sum_{E\in\{0,1\}^{m\!\times\!n}}t^{\sum_{ij}\!E_{ij}} \sum_{x\in\{-1,1\}^m\!,\,y\in\{-1,1\}^n}\! \Big(\!\prod_i\!x_i^{\sum_j\!E_{ij}}\!\Big) \Big(\!\prod_j\!y_j^{\sum_i\!E_{ij}}\!\Big)$\\
$\overset{(4)}{=} \frac{(1+t)^{m+n}}{2^{m+n}}\! \sum_{x\in\{-1,1\}^m\!,\,y\in\{-1,1\}^n}\! \sum_{E\in\{0,1\}^{m\!\times\!n}}\prod_{i,j}(x_iy_jt)^{E_{ij}}$\\
$\overset{(5)}{=} \frac{(1+t)^{m+n}}{2^{m+n}}\! \sum_{x\in\{-1,1\}^m\!,\,y\in\{-1,1\}^n}\! \prod_{i,j}(1\!+\!x_iy_jt)$\\
$\overset{(6)}{=} \frac{(1+t)^{m+n}}{2^{m+n}}\sum_{i,j}\!\binom{m}{i}\binom{n}{j}(1\!+\!t)^{(m\!-\!i)(n\!-\!j)+ij}(1\!-\!t)^{(m\!-\!i)j+(n\!-\!j)i}$.\\
\end{longtable}\vspace{-1mm}
\noindent Equation $(1)$ holds, since \smash{$\prod_i\!\big(1\!+\!(-1)^{\sum_j\!E_{ij}}\big)\prod_j\!\big(1\!+\!(-1)^{\sum_i\!E_{ij}}\big)=0$} whenever the sum of a row $\sum_j\!E_{ij}$ or column $\sum_i\!E_{ij}$ is odd. Furthermore, $(2)$ is by the multibinomial theorem, $(3)$ is just distributivity, $(4)$ is just rearranging the sums and products. Equation $(5)$ holds, because writing $\prod_{i,j}(1\!+\!x_iy_jt)$ as a sum, there are $2^{mn}$ choices (= $0/1$-matrix) of picking $1$ or $x_iy_jt$ as a factor. Equation $(6)$ holds, since if we choose $i$ out of $m$ minuses for $x$ and $j$ out of $n$ minuses for $y$, then $x_iy_j\!=\!1$ for $(m\!-\!i)(n\!-\!j)+ij$ pairs and $x_iy_j\!=\!-1$ for $(m\!-\!i)j+(n\!-\!j)i$ pairs.
\end{proof}\vspace{1mm}

\begin{Rmk}\label{Konvalinka} We were pleasantly surprised that such beautiful and striking formulas can be derived. Moreover, by using the transfer matrix method \cite[4.7, p.500]{citeStanleyEC}, Matjaž Konvalinka obtained an alternative formula: \vspace{-2mm}
\begin{longtable}[c]{r@{$\;\;=\;\;$}l}
$\HP_{m,n,2}$ & $\frac{(1+t)^{m+n}}{2^n}\sum_k\! \binom{n}{k} \Big( \sum_i\! \big(-\!\binom{n}{2i} + 2\sum_j\!\binom{k}{2j}\binom{n-k}{2i-2j}\big) t^{2i}\Big)^{\!\!m}$.
\end{longtable}\vspace{-2mm}
\noindent In fact, the transfer matrix of this 2-dimensional sequence $\HP_{m,n,2}$ possesses a number of pleasant properties (symmetry, orthogonal set of eigenvectors, etc).
\end{Rmk}\vspace{1mm}

\begin{Rmk} In particular, the first few generating functions are\vspace{-2mm}
\begin{longtable}[c]{l}
$\HP_{2,n,2}=\frac{(1+t)^{n+2}\big((1-t^2)^n+(1+t^2)^n\big)}{2}$,\\
$\HP_{3,n,2}=\frac{(1+t)^{n+3}\big(3(1-t^2)^n+(1+3t^2)^n\big)}{4}$,\\
$\HP_{4,n,2}=\frac{(1+t)^{n+4}\big(3(1-t^2)^{2n}+4(1-t^4)^n+(1+6t^2+t^4)^n\big)}{8}$,\\
$\HP_{5,n,2}=\frac{(1+t)^{n+5}\big(10(1-t^2)^{2n}+5(1+2t^2-3t^4)^n+(1+10t^2+5t^4)^n\big)}{16}$.
\end{longtable}\vspace{-2mm}
\noindent Since these are not of the form $\prod_i(1\!+\!t^{d_i})$, it follows that the cohomology is not isomorphic to any exterior polynomial algebra (it is a proper quotient, there are nontrivial relations). In fact, by \ref{cupproduct} the cup product corresponds to the wedge product, so multiplying two generators (= 0/1 matrices $E$ and $E'$) in $H^\ast(\frak{gl}_{m+n}^\preceq;\Z_p)$ gives either $0$ (when they both have value $1$ at some same entry: $\exists i,j\!:\, E_{ij}\!=\!1\!=\!E'_{ij}$) or their sum (when the matrices have disjoint supports: $\forall i,j\!:\, E_{ij}E'_{ij}\!=\!0$). An interesting problem is to obtain a presentation, or at least the minimal number of generators for the cohomology algebra: how many even $0/1$-matrices does one need so that every even $0/1$-matrix is a product (=disjoint sum) of these?
\end{Rmk}

\vspace{4mm}
\section{Posets of height 2}\label{HGT2}  In this section, all posets are assumed to have height $2$. A wedge in $C_{\star\!,p}(\frak{gl}_n^\prec)$ in this case corresponds to a subset of edges of the Hasse diagram $\Gamma_{\!\preceq}$, together with the set of all directed paths of length $2$ (two paths are identified when they have the same endpoints). For a given family of Hasse diagrams indexed by $n\!\in\!\N$, we denote by $\HP_{n,p}$ the generating function of the (co)homology of $\frak{gl}^\preceq(\Z_p)$. \vspace{2mm}

\subsection{Fork posets} Let our poset $\preceq$ be given by the Hasse diagram below left. Thus $\dim \frak{gl}^\preceq_{2n+1}=5n\!+\!1$. A specific faithful representation is shown below right.\vspace{-2mm}
$$\begin{tikzpicture}[baseline=(m.center)]\matrix (m) [matrix of math nodes, row sep=13pt, column sep=5pt]{
c_1 &c_2 & \ldots & c_{n\!-\!1}& c_n\\
b_1 &b_2 & \ldots & b_{n\!-\!1}& b_n\\
&&a&&\\ };
\draw(m-3-3)--(m-2-1); \draw(m-3-3)--(m-2-2); \draw(m-3-3)--(m-2-3); \draw(m-3-3)--(m-2-4); \draw(m-3-3)--(m-2-5);
\draw(m-2-1)--(m-1-1); \draw(m-2-2)--(m-1-2); \draw(m-2-3)--(m-1-3); \draw(m-2-4)--(m-1-4); \draw(m-2-5)--(m-1-5); \end{tikzpicture}\hspace{8mm}\frak{gl}^\preceq\!=\!\bigg\{ \raisebox{-6mm}{\includegraphics[width=0.11\textwidth]{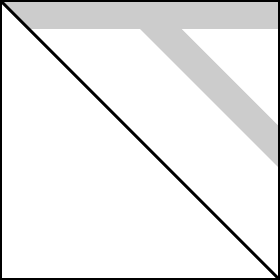}}\bigg\}\vspace{-2mm}$$

\begin{Prp} $\HP_{\!n,2}\!=\!(1\!+\!t)^{2n+1}(1\!+\!t^3)^n$ and $\HP_{\!n,p}\!=\!(1\!+\!t)^{2n+1}$ for $p\!\geq\!3$.
\end{Prp}
\begin{proof} Only index $a$ can have weight more than $\pm2$, so $C_{\star,p}$ is empty for $p\!\geq\!3$. Every $b_i$ has weight $0$ or $\pm1$, so $C_{\star,2}$ is spanned by $\big\{\bigwedge_{i\in\sigma}\!e_{ab_i}e_{b_ic_i}e_{ac_i};\,\sigma\!\subseteq\![n]\big\}$, and the boundary of these wedges is $0$, hence $\HP_{C_{\star,2}}=\sum_{k\in\N}\binom{n}{k}t^{3k}=(1\!+\!t^3)^{n}$.
\end{proof}\vspace{2mm}

\subsection{Umbrella posets} Let our poset $\preceq$ be given by the Hasse diagram below left. Thus $\dim \frak{gl}^\preceq_{n+2}=3n\!+\!3$. An example of a faithful representation is below right.\vspace{-2mm}
$$\begin{tikzpicture}[baseline=(m.center)]\matrix (m) [matrix of math nodes, row sep=13pt, column sep=5pt]{
c_1 &c_2 & \ldots & c_{n\!-\!1}& c_n\\
&&b&&\\
&&a&&\\ };
\draw(m-3-3)--(m-2-3); \draw(m-2-3)--(m-1-1); \draw(m-2-3)--(m-1-2); \draw(m-2-3)--(m-1-3); \draw(m-2-3)--(m-1-4); \draw(m-2-3)--(m-1-5); \end{tikzpicture}\hspace{8mm}\frak{gl}^\preceq\!=\!\bigg\{ \raisebox{-6mm}{\includegraphics[width=0.11\textwidth]{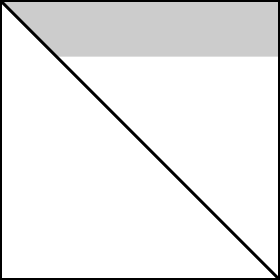}}\bigg\}\vspace{-2mm}$$
Notice that $\frak{gl}^\prec$ is the family of 2-step nilpotent Lie algebras from \cite{citeArmstrongCairnsJessupEBNFNLA}.\vspace{1mm}

\begin{Prp} $\HP_{\!n,2}\!=\!\frac{(1+t)^{n+3}(1+t^2)^n+(1+t)^{n+1}(1-t^2)^{n+1}}{2}$.
\end{Prp}
\begin{proof} Only indices $a$ and $b$ can have weight more than $\pm2$, so $C_{\star,p}$ is empty for $p\!\geq\!3$. In a wedge $v\!\in\!C_{\star,2}$, an index $c_i$ has weight $0$ (iff $e_{ac_i},e_{bc_i}\!\notin\!v$) or $2$ (iff $e_{ac_i},e_{bc_i}\!\in\!v$), and for any choice $\sigma\!\subseteq\![n]$ that determines which $c_i$ appear in the wedge, $b$ has weight $-|\sigma|$ (iff $e_{ab}\!\notin\!v$) or $1\!-\!|\sigma|$ (iff $e_{ab}\!\in\!v$). Therefore the complex $C_{\star,2}$ admits a basis $\big\{e''_\sigma\!:=\!\bigwedge_{i\in\sigma}\!e_{ac_i}e_{bc_i};\,\sigma\!\subseteq\![n], |\sigma|\!\in\!2\N\big\} \!\cup\! \big\{e'_\sigma\!:=\!e_{ab}\bigwedge_{i\in\sigma}\!e_{ac_i}e_{bc_i};\,\sigma\!\subseteq\![n], |\sigma|\!+\!1\!\in\!2\N\big\}$, with $\partial(e''_\sigma)\!=\!0$ and $\partial(e'_\sigma)\!=\!0$. Thus \vspace{-2mm}
\begin{longtable}[c]{l}
\;\;$\dim H_{2k}(\frak{gl}^\preceq;\Z_2)=|\{\sigma\!\subseteq\![n];\, |\sigma|\!=\!k\!\in\!2\N\}|$ \;and\\
$\dim H_{2k+1}(\frak{gl}^\preceq;\Z_2)=|\{\sigma\!\subseteq\![n];\, |\sigma|\!=\!k\!\in\!2\N\!+\!1\}|$.
\end{longtable}\vspace{-2mm}
\noindent Since $\sum_{k\in 2\N}\!t^k\!\binom{n}{k} \!+\! \sum_{k\in 2\N+1}\!t^k\!\binom{n}{k} \!=\!(1\!+\!t)^n$\! and $\sum_{k\in 2\N}\!t^k\!\binom{n}{k} \!-\! \sum_{k\in 2\N+1}\!t^k\!\binom{n}{k} \!=\!(1\!-\!t)^n$\!, we have $\sum_{k\in 2\N}\!t^k\!\binom{n}{k} \!=\! \frac{(1+t)^n+(1-t)^n}{2}$ and $\sum_{k\in 2\N+1}\!t^k\!\binom{n}{k} \!=\! \frac{(1+t)^n-(1-t)^n}{2}$. Therefore\vspace{-1mm}
\begin{longtable}[c]{r@{\;$=$\;}l}
$\HP_{\!C\!_{\star\!,2}(\frak{gl}_n^\prec)}$ &$\sum_{k\in2\N}t^{2k}\binom{n}{k} + \sum_{k\in 2\N+1}t^{2k+1}\binom{n}{k}$\\[2pt]
      &$\frac{(1+t^2)^n+(1-t^2)^n}{2}+t\frac{(1+t^2)^n-(1-t^2)^n}{2} = \frac{(1+t)(1+t^2)^n+(1-t)(1-t^2)^n}{2}$.
\end{longtable}\vspace{-1mm} \noindent Then $\HP_{n,2}=(1\!+\!t)^{n+2}\HP_{\!C\!_{\star\!,2}}=\frac{(1+t)^{n+3}(1+t^2)^n+(1+t)^{n+1}(1-t^2)^{n+1}}{2}$.
\end{proof}\vspace{1mm}

\begin{Crl} If $\Z_2\!\leq\!R$, then there is an isomorphism of graded $R$-algebras\vspace{-1mm}
\begin{longtable}[c]{l}
$H^\ast(\frak{gl}^\preceq;R)\cong\Lambda\big[x_a,x_b,x_i,y_\omega,z_i;\, i\!\in\![n],\omega\!\in\!\binom{[n]}{2}\big]/I$,\\
$I=\langle y_{\{i,j\}}y_{\{j,k\}},y_{\{i,j\}}y_{\{k,l\}}\!-\!y_{\{i,k\}}y_{\{j,l\}},y_{\{i,j\}}z_i,y_{\{i,j\}}z_k\!-\!y_{\{i,k\}}z_j,z_iz_j\rangle$,
\end{longtable}\vspace{-2mm}
\noindent where $x_a\!\equiv\!e_{aa}$, $x_b\!\equiv\!e_{bb}$, $x_i\!\equiv\!e_{c_ic_i}$, $y_{\{i,j\}}\!\equiv\!e_{ac_i}e_{bc_i}e_{ac_j}e_{bc_j}$, $z_i\!\equiv\!e_{ab}e_{ac_i}e_{bc_i}$.
\end{Crl}
\noindent If $\Z_p\!\leq\!R$ with $p\!\geq\!3$, then $H^\ast(\frak{gl}^\preceq;R)\cong\Lambda[x_a,x_b,x_i]$.
\begin{proof} In the proof above, we saw that the $p$-complex is generated by the wedges $e''_\sigma\!=\! \bigwedge_{i\in\sigma}\!e_{ac_i}e_{bc_i}$ with $|\sigma|\!\in\!2\N$ and $e'_\sigma\!=\! e_{ab}\bigwedge_{i\in\sigma}\!e_{ac_i}e_{bc_i}$ with $|\sigma|\!\in\!2\N\!+\!1$. Hence $H^\ast(\frak{gl}^\preceq;R)$ is generated by the duals of all $e_\tau e''_\sigma$ and $e_\tau e'_\sigma$, where $e_\tau$ is the wedge of diagonals, as in \cite[5.10]{citeLampretVavpeticCLAAMT}. By \ref{cupproduct}, cup products correspond to wedges. Thus $e_\tau$ is the product of diagonals, $e''_\sigma$ is a (non-unique) product of those $e''_\omega$ with $|\omega|\!=\!2$ (denoted by $y_\omega$), and $e'_\sigma$ is a (non-unique) product of an \smash{$e'_{\{i\}}$} (denoted by $z_i$) and \smash{$e''_{\sigma\setminus\{i\}}$}. In this way, our cohomology algebra can be viewed as the subalgebra of the whole exterior algebra $\Lambda[e_{xy};\, x\!\preceq\!y]$. The relations come from the property that a wedge behaves as a set, and is zero whenever it contains duplicate elements.
\end{proof}\vspace{2mm}

\subsection{Diamond posets} Let our poset $\preceq$ be given by the Hasse diagram below left. Thus $\dim \frak{gl}^\prec_{n+2}=3n\!+\!3$. A specific embedding into $\frak{gl}_n$ is shown below right.\vspace{-2mm}
$$\begin{tikzpicture}[baseline=(m.center)]\matrix (m) [matrix of math nodes, row sep=13pt, column sep=5pt]{
&&c&\\
b_1 &b_2 & \ldots & b_{n\!-\!1}& b_n\\
&&a&\\ };
\draw(m-3-3)--(m-2-1); \draw(m-3-3)--(m-2-2); \draw(m-3-3)--(m-2-3); \draw(m-3-3)--(m-2-4); \draw(m-3-3)--(m-2-5);
\draw(m-2-1)--(m-1-3); \draw(m-2-2)--(m-1-3); \draw(m-2-3)--(m-1-3); \draw(m-2-4)--(m-1-3); \draw(m-2-5)--(m-1-3); \end{tikzpicture}\hspace{8mm}\frak{gl}^\preceq\!=\!\bigg\{ \raisebox{-6mm}{\includegraphics[width=0.11\textwidth]{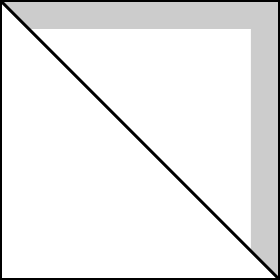}}\bigg\}\vspace{-2mm}$$
Notice that $\frak{gl}^\prec$ is the family of Heisenberg Lie algebras, whose homology was determined in \cite{citeSantharoubaneCHLA} and \cite{citeCairnsJamborCHLAOFFC}. We compute $H_\ast$ of the solvable analog (added diagonals).\vspace{1mm}

\begin{Prp}\label{diamondZ2} $\HP_{\!n,2}=(1\!+\!t)^{n+2}(1\!+\!t^3)\frac{(1+t^2)^{n\!-\!1}+(1-t^2)^{n\!-\!1}}{2}$.
\end{Prp}
\begin{proof} In a wedge $v\!\in\!C_\star(\frak{gl}^\prec)$, indices $b_i$ have weight either $\pm1$ (iff $e_{ab_i}\!\notin\!v\!\ni\!e_{b_ic}$ or $e_{ab_i}\!\in\!v\!\not\ni\!e_{b_ic}$) or $0$, so if $v$ is contained in $C_{\star,p}$, then every $b_i$ appears twice or not at all in $v$. For any $\sigma\!\subseteq\![n]$ denote $e''_\sigma\!=\! \bigwedge_{i\in\sigma}\!e_{ab_i}e_{b_ic}$ and $e'_\sigma\!=\! e_{ac}\bigwedge_{i\in\sigma}\!e_{ab_i}e_{b_ic}$. Then $\{e''_\sigma;\, |\sigma|\!\in\!p\N\} \!\cup\! \{e'_\sigma;\, |\sigma|\!+\!1\!\in\!p\N\}$ is a basis for $C_{\star,p}(\frak{gl}^\prec)$, with $\partial(e''_\sigma) \!=\! -\sum_{i\in\sigma}\!e'_{\sigma\setminus\{i\}}$ and $\partial(e'_\sigma)\!=\!0$. We define a Morse matching \smash{$\mathcal{M}\!=\!\big\{ e''_{\sigma\cup\{n\}} \!\!\rightarrow\! e'_\sigma;\, n\!\notin\!\sigma\big\}$}. Critical vertices are \smash{$\mathring{\mathcal{M}}\!=\! \{e''_\sigma;\, n\!\notin\!\sigma\} \!\cup\! \{e'_\sigma;\, n\!\in\!\sigma\}$}. Zig-zag paths come in pairs: \vspace{-1mm} $$\xymatrix@R=0.5mm@C=7mm{e''_\sigma\ar[r]&    e'_{\sigma\setminus\{i\}}\\
                       e_{(\sigma\setminus\{i\})\cup\{n\}} \ar[ru]|-{\scriptscriptstyle\mathcal{M}}\ar[r]&e_{(\sigma\setminus\{i,j\})\cup\{n\}}}\text{ \;\;\;and\;\;\; }
                       \xymatrix@R=0.5mm@C=7mm{e''_\sigma\ar[r]&    e'_{\sigma\setminus\{j\}}\\
                       e_{(\sigma\setminus\{j\})\cup\{n\}} \ar[ru]|-{\scriptscriptstyle\mathcal{M}}\ar[r]&e_{(\sigma\setminus\{i,j\})\cup\{n\}}}.\vspace{-1mm}$$
Therefore the boundary is $\mathring{\partial}(e''_\sigma)\!=\! \sum_{i,j\in\sigma}2\,e'_{(\sigma\setminus\{i,j\})\cup\{n\}}$. Over $\Z_2$, this is zero, so \vspace{-2mm}
\begin{longtable}[c]{l}
\;\;$\dim H_{2k}(\frak{gl}^\preceq;\Z_2)=|\{\sigma\!\subseteq\![n\!-\!1];\, |\sigma|\!=\!k\!\in\!2\N\}|$ \;and\\
$\dim H_{2k+1}(\frak{gl}^\preceq;\Z_2)=|\{\sigma\!\subseteq\![n\!-\!1];\, |\sigma|\!=\!k\!-\!1\!\in\!2\N\}|$.
\end{longtable}\vspace{-2mm}
\noindent Every $e''_\sigma$ contributes $t^{2|\sigma|}$ and every $e'_{\sigma\cup\{n\}}$ contributes $t^{2|\sigma|+3}$. Hence \vspace{-1mm}
\begin{longtable}[c]{r@{\;$=$\;}l}
$\HP_{\!C\!_{\star\!,2}(\frak{gl}_n^\prec)}$ &$\sum_{k\in2\N}t^{2k}\binom{n-1}{k} + \sum_{k\in2\N}t^{2k+3}\binom{n-1}{k}$\\
      &$(1\!+\!t^3)\sum_{k\in2\N}(t^2)^k\binom{n-1}{k} = (1\!+\!t^3)\frac{(1+t^2)^{n\!-\!1}+(1-t^2)^{n\!-\!1}}{2}$.
\end{longtable}\vspace{-1mm} \noindent Over $\Z_p$ with $p\!>\!2$, additional matchings are required, but things become messy.
\end{proof}\vspace{2mm}

\begin{Crl} If $\Z_2\!\leq\!R$, then there is an isomorphism of graded $R$-algebras\vspace{-1mm}
\begin{longtable}[c]{l}
$H^\ast(\frak{gl}^\preceq;R)\cong\Lambda[x_a,x_i,x_c,y_\omega,z;\, i\!\in\![n],\omega\!\in\!\binom{[n\!-\!1]}{2}]/\langle y_{\{i,j\}}y_{\{j,k\}}\rangle,$
\end{longtable}\vspace{-2mm}
\noindent where $x_a\!\equiv\!e_{aa}$, $x_i\!\equiv\!e_{b_ib_i}$, $x_c\!\equiv\!e_{cc}$, $y_{\{i,j\}}\!\equiv\! e''_{\{i,j\}}\!+\! e''_{\{i,n\}}\!+\! e''_{\{j,n\}}$ $z\!\equiv\!e_{ac}e_{ab_n}e_{b_nc}$, $e''_{\{i,j\}}\!=\!e_{ab_i}e_{b_ic}e_{ab_j}e_{b_jc}$ (the length of a wedge determines the degree).
\end{Crl}
\begin{proof} In the proof above, we saw that the $2$-complex is spanned by certain wedges $e''_\sigma\!=\! \bigwedge_{i\in\sigma}\!e_{ab_i}e_{b_ic}$ and $e'_\sigma\!=\! e_{ab}\!\bigwedge_{i\in\sigma}\!e_{ab_i}e_{b_ic}$. For $\overline{e}''_\sigma\!:=\! e''_\sigma\!+\!\sum_{i\in\sigma}\!e''_{(\sigma\setminus\{i\})\cup\{n\}}$, the set
$$\big\{\overline{e}''_\sigma, e''_{\tau\cup\{n\}};\, \sigma,\tau\!\subseteq\![n\!\!-\!\!1],\, |\sigma|,|\tau|\!+\!1\!\in\!2\N\big\} \cup \big\{e'_\sigma;\, \sigma\!\subseteq\![n],\, |\sigma|\!\in\!2\N\!+\!1\big\}$$ is a new basis, and after the same matching as above, the critical vertices are
$$\big\{\overline{e}''_\sigma;\, \sigma\!\subseteq\![n\!\!-\!\!1],\, |\sigma|\!\in\!2\N\big\} \cup \big\{e'_{\sigma\cup\{n\}};\, \sigma\!\subseteq\![n\!\!-\!\!1],\, |\sigma|\!\in\!2\N\big\},$$
but this time there are no zig-zag paths, since $\partial(\overline{e}''_\sigma)\!=\!0$. In this way, we obtained a subcomplex with zero boundaries that generates homology. Let $\overline{\chi}''_\sigma$ and $\chi'_{\sigma\cup\{n\}}$ be the dual wedges that generate cohomology, so that $\overline{\chi}''_\sigma(\overline{e}''_\sigma)\!=\! \overline{\chi}''_\sigma(e''_\sigma)\!=\!1$. Since $\overline{e}''_\sigma \overline{e}''_\tau\!=\! \icases{\overline{e}''_{\sigma\cup\tau}}{\!\!\text{if }\sigma\cap\tau=\emptyset}{0}{\!\!\text{if }\sigma\cap\tau\neq\emptyset}{\scriptstyle}{-2pt}{\big}$ and $\overline{e}''_\sigma e'_{\tau\cup\{n\}}\!=\! \icases{e'_{\sigma\cup\tau\cup\{n\}}}{\text{\!\!if }\sigma\cap\tau=\emptyset}{0}{\!\!\text{if }\sigma\cap\tau\neq\emptyset}{\scriptstyle}{-2pt}{\big}$ and $e'_\sigma e'_\tau=0$, these relations also hold for cup products of dual wedges, by \ref{cupproduct}. Thus $H^\ast(\frak{gl}^\preceq;R)$ is isomorphic to the subalgebra of $\Lambda[e_{xy};\, x\!\preceq\!y]$ generated by all $e_{ii},\overline{e}''_{\{i,j\}},e'_{\{n\}}$.
\end{proof}\vspace{2mm}

\begin{Crl}\label{diamondZ3} For $F(t)\!=\!\frac{t(1+t)^{n+1}}{1+t+t^2} \big(1\!-\!(\frac{t}{1+2t+t^2})^{\lceil n/2\rceil}\big)$ and $\epsilon\!=\!e^{2\pi\mathbf{i}/3}$ we have
$$\textstyle{\HP_{n,3}\!=\!\frac{(1+t)^{n+2}}{3}\sum_{i\in[3]}(1\!+\!\epsilon^it)(1\!+\!\epsilon^it^2)^n\!-\!\frac{t+1}{t}F(\epsilon^it^2).}$$
\end{Crl}\vspace{1mm}
\begin{proof} By the derivation of \ref{diamondZ2}, for any prime $p$ the generating function over $\Z_p$ is\vspace{-1mm}
\begin{longtable}[c]{l}
$(1\!+\!t)^{n+2}\Big(\!\sum_{k\in p\N}\!\big(\binom{n}{k}\!-\!\mathrm{rank}_{\Z_p}\!\partial_k\big)t^{2k}+
\sum_{k\in p\N-1}\!\big(\binom{n}{k}\!-\!\mathrm{rank}_{\Z_p}\!\partial_{k+1}\big)t^{2k+1}\Big)$\\
$=(1\!+\!t)^{n+2}\!\sum_{k\in p\N}\big(\binom{n}{k}\!-\!\mathrm{rank}_{\Z_p}\!\partial_k\big)t^{2k}+\big(\binom{n}{k\!-\!1}\!-\!\mathrm{rank}_{\Z_p}\!\partial_{k}\big)t^{2k\!-\!1}$\\
$=(1\!+\!t)^{n+2}\!\sum_{k\in p\N}\binom{n}{k}t^{2k}\!+\!\binom{n}{k\!-\!1}t^{2k\!-\!1}\!-\!\frac{t+1}{t}t^{2k}\mathrm{rank}_{\Z_p}\!\partial_k$\\
$=\frac{(1+t)^{n+2}}{p}\sum_{i\in[p]}(1\!+\!\epsilon^it^2)^n+ t^{-1}\epsilon^it^2(1\!+\!\epsilon^it^2)^n- \frac{t+1}{t}F(\epsilon^it^2)$ by \ref{DFT}\\
$=\frac{(1+t)^{n+2}}{p}\sum_{i\in[p]}(1\!+\!\epsilon^it)(1\!+\!\epsilon^it^2)^n- \frac{t+1}{t}F(\epsilon^it^2)$,
\end{longtable}\vspace{-2mm}
\noindent where $\partial_k$ is a combinatorial $0/1$-matrix, namely the `subsets incidence' matrix of size $\binom{n}{k\!-\!1}\!\times\!\binom{n}{k}$\!, where $\sigma\tau$-th entry is $1$ iff $\sigma\!\subseteq\!\tau$. By \cite{citeLinialRothschildIMSRF}, $\mathrm{rank}_{\Z_3}\partial_k\!=\!\sum_j\!\binom{n-2j-1}{k-j-1}$, so \vspace{-1mm}
\begin{longtable}[c]{l}
$F(t)= \sum_k\mathrm{rank}_{\Z_3}\!\partial_k\,t^k= \sum_{k,j\in\N}\binom{n-2j-1}{k-j-1}t^k= \sum_{j,k\in\N}\binom{n-2j-1}{k}t^{k+j+1}$\\
$=t\sum_{j}\!t^j\sum_k\!\binom{n-2j-1}{k}t^k =t\sum_{0\leq n-2j-1}\!t^j(1\!+\!t)^{n-2j-1}$\\
$=t(1\!+\!t)^{n\!-\!1}\sum_{j\leq (n-1)/2}\!\big(\frac{t}{(1+t)^2}\big)^j =t(1\!+\!t)^{n\!-\!1}\big(\sum_{j\in\N}\!-\!\sum_{j\geq n/2}\big)$\\
$=t(1\!+\!t)^{n\!-\!1}\big(1\!-\!(\frac{t}{1+2t+t^2})^{\lceil n/2\rceil}\big)/\big(1\!-\!\frac{t}{1+2t+t^2}\big)=\frac{t(1+t)^{n+1}}{1+t+t^2} \big(1\!-\!(\frac{t}{1+2t+t^2})^{\lceil n/2\rceil}\big)$.
\end{longtable}\vspace{-2mm}
\noindent For cases $p\!\geq\!5$ we feel intimidated to try computing the generating function.
\end{proof}

\vspace{4mm}
\section{Distributions of coefficients}
In this section, we observe how the coefficients of the generating functions (= dimensions of homologies = Betti numbers) we computed are allocated. Some research in this direction was done in \cite{citePouseeleBNBNLA}, where it was shown that the coefficients for nilpotent Lie algebras are often unimodal (i.e. ascending and then descending) or have an M-shape. Here we display much more exotic examples.\\[2mm]

For the diamond poset, the weighted coefficients (divided by maximum) of $\HP_{n,2}$ in \ref{diamondZ2} for $50\!\leq\!n\!\leq\!150$ with $n\!\in\!10\N$ are plotted in the following picture:\\
\includegraphics[width=0.99\textwidth]{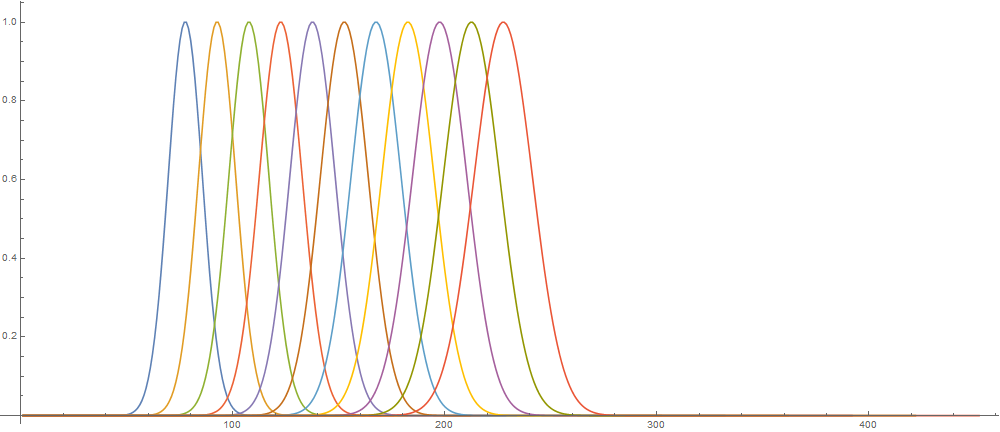} \\[2mm]

For the diamond poset, the coefficients of $\HP_{n,3}$ in \ref{diamondZ3} for $1\!\leq\!n\!\leq\!14$ are plotted in the following pictures (though for large $n$ the distribution becomes binomial):\\
\begin{longtable}[c]{cc}
\includegraphics[width=0.47\textwidth]{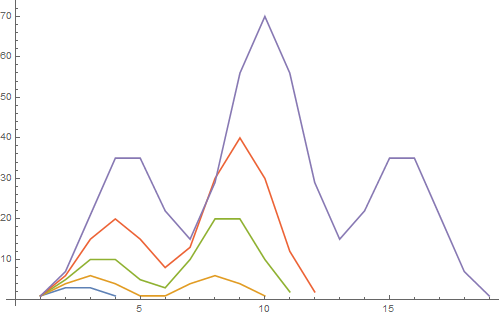} &\includegraphics[width=0.47\textwidth]{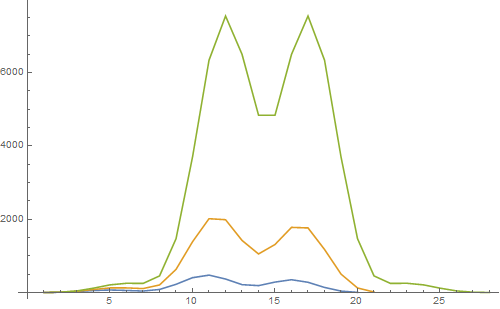}\\
\includegraphics[width=0.47\textwidth]{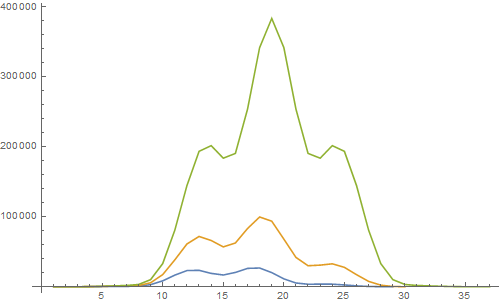} &\includegraphics[width=0.47\textwidth]{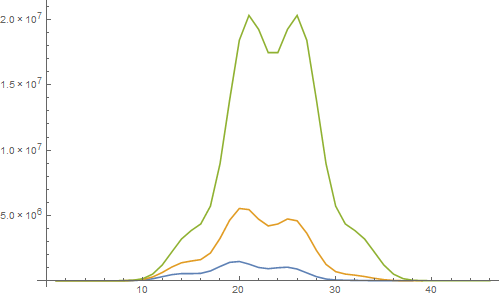}
\end{longtable}

For the complete bipartite posets, the coefficients of $\HP_{p,n,p}$ in \ref{pnp} for $n\!=\!200$ and $p\!=\!2,3,5,7$ (from left to right) are plotted in the following picture:\\
\includegraphics[width=0.99\textwidth]{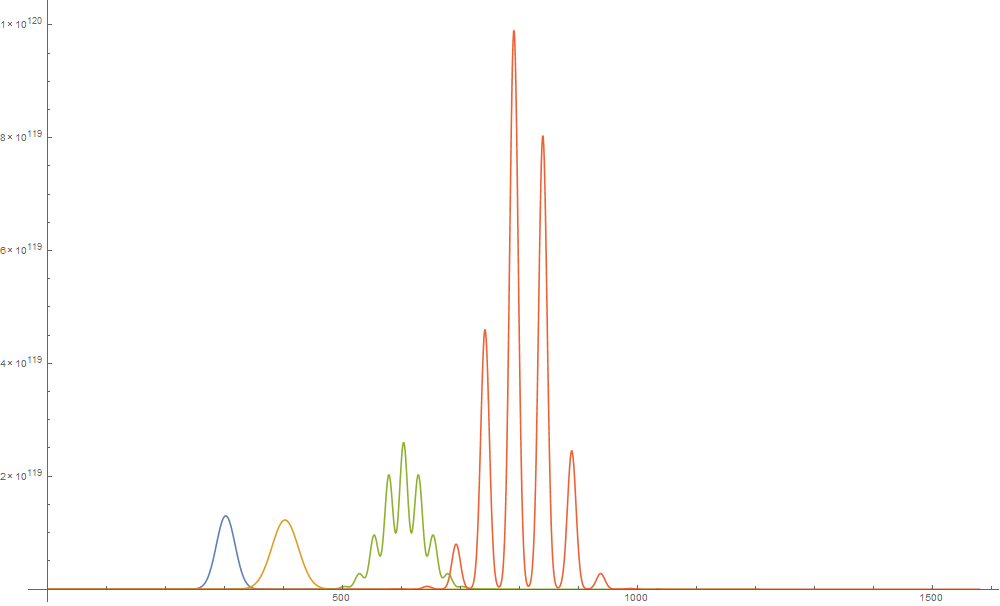}\\
\noindent We suspect this provides an affirmative answer to the second question in \cite[p.86]{citePouseeleBNBNLA}.

\vspace{4mm}
\section{Afterword}
\subsection{Conclusion} We've seen, that AMT enables graph theory and combinatorics to come to the aid of homological algebra. It was unexpected that so much rich problems come from really simple posets of height only $1$ and $2$.
\par At present, very few Hilbert-Poincar\'{e} series for the (co)homology of Lie algebras are known (we provided six new families), and fewer still are the computed cup products in the literature (we provided three new families, which are quite interesting). The formulas seem to be nicest over rings of characteristic two, partly due to the fact that the primitive root of unity $e^{2\pi\mathbf{i}/p}\!\in\!\C$ is real iff $p\!=\!2$. 
\par Calculations of $H^\ast(\frak{gl}^\preceq;\Z_p)$ for posets of larger height are also possible, but then it is easier to compute over large characteristics $p$. The use of the $p$-complex $C_{\star,p}(\frak{gl}_n^\prec)$ for the right kind of posets $\preceq$ and primes $p$ can reduce the calculation of cohomology to mere enumerative combinatorics, which is the goal of AMT. 
\subsection{Acknowledgment} This research was supported by the Slovenian Research Agency grants P1-0292-0101, J1-5435-0101, J1-6721-0101, BI-US/12-14-001.
\par We wish to thank Richard Stanley for generously providing the beautiful formula \ref{Stanley} and its elegant proof, and Matjaž Konvalinka for producing an alternative formula \ref{Konvalinka}, as well as many pieces of advice in combinatorial investigations.

\end{document}